\documentclass[10pt]{amsart}
\usepackage{latexsym}
\usepackage{amsmath}
\usepackage{amssymb}
\usepackage{mathrsfs}
\usepackage{graphicx}
\usepackage{color}
\usepackage{pgfpages}
\usepackage{ifthen}
\usepackage{leftidx,tensor}
\usepackage[T1]{fontenc}
\usepackage[latin1]{inputenc}
\usepackage{mathtools}
\usepackage{comment}
\usepackage{dsfont}
\usepackage[nocompress]{cite}

\usepackage[shortlabels]{enumitem}
\usepackage{aliascnt}
\usepackage[bookmarks=true,pdfstartview=FitH, pdfborder={0 0 0}, colorlinks=true,citecolor=red, linkcolor=blue]{hyperref}
\usepackage{bbm}
\usepackage{nicefrac}
\theoremstyle{plain}
\newtheorem{thm}{Theorem}[section]
\newaliascnt{cor}{thm}
\newaliascnt{prop}{thm}
\newaliascnt{lem}{thm}
\newtheorem{cor}[cor]{Corollary}
\newtheorem{prop}[prop]{Proposition}
\newtheorem{lem}[lem]{Lemma}
\aliascntresetthe{cor}
\aliascntresetthe{prop}
\aliascntresetthe{lem} 
\theoremstyle{definition}
\newaliascnt{defn}{thm}
\newaliascnt{asu}{thm}
\newaliascnt{con}{thm}
\aliascntresetthe{defn}
\aliascntresetthe{asu}
\aliascntresetthe{con}
\newcounter{stp}
\newcounter{stpi}
\newcounter{stpci}
\newcounter{stpiii}

\newtheorem{step}[stp]{Step}

\theoremstyle{thm}
\newaliascnt{rem}{thm}
\newaliascnt{exa}{thm}
\newaliascnt{masu}{thm}
\newaliascnt{nota}{thm}
\newaliascnt{sett}{thm}
\newtheorem{rem}[rem]{Remark}

\aliascntresetthe{rem}
\aliascntresetthe{exa}
\aliascntresetthe{masu}
\aliascntresetthe{nota}
\aliascntresetthe{sett}
\numberwithin{equation}{section}

\setlist[enumerate]{font = \normalfont}

\newcommand {\N}	{\mathrm{N}}

\newcommand {\R}	{\mathbb{R}}
\newcommand {\C}	{\mathbb{C}}
\newcommand {\E}	{\mathbb{E}}
\newcommand {\F}	{\mathbb{F}}

\newcommand {\T}	{\mathbb{T}}

\renewcommand{\d}{\, \mathrm{d}}

\DeclareMathOperator{\divH}{div_{\H}}

\newcommand{\D}{\mathrm{D}}
\newcommand{\rN}{\mathrm{N}}
\renewcommand{\H}{\mathrm{H}}

\newcommand{\per}{\mathrm{per}}

\newcommand{\sigmabar}{\bar{\sigma}}

	\newcommand{\dk}[1]{\partial_{#1}}
	\newcommand{\dt}{\dk{t}} 
	\newcommand{\dz}{\dk{z}} 
	
	\newcommand{\eps}{\varepsilon}
	\renewcommand{\phi}{\varphi}
	
	\renewcommand{\bar}[1]{\overline{#1}}
	\newcommand{\vbar}{\bar{v}}

	\renewcommand{\div}{\mathrm{div} \, }
	\newcommand{\nablaH}{\nabla_{\H}}
	\newcommand{\DeltaH}{\Delta_{\H}}

	\newcommand{\rC}{\mathrm{C}}
	\newcommand{\rL}{\mathrm{L}}
	\newcommand{\rW}{\mathrm{W}}
	\newcommand{\rH}{\H}
	\newcommand{\rB}{\mathrm{B}}

	\newcommand{\rLq}{\rL^q}
	\newcommand{\rLp}{\rL^p}

	\newcommand{\rX}{\mathrm X}

	\newcommand{\etz}{\color{black}}

\title[Dynamic boundary conditions with noise for an EBM coupled to geophysical flows]{Dynamic boundary conditions with noise for an energy balance model coupled to geophysical flows
  }

\author{Gianmarco Del Sarto}
\address{Technische Universit\"{a}t Darmstadt\\
Fachbereich Mathematik\\
	Schlossgartenstr.\ 7\\
	64289 Darmstadt\\
	Germany}
\email{delsarto@mathematik.tu-darmstadt.de}

\author{Matthias Hieber}
\address{Technische Universit\"at Darmstadt\\
	Fachbereich Mathematik\\
	Schlossgartenstr.\ 7\\
	64289 Darmstadt\\
	Germany}
\email{hieber@mathematik.tu-darmstadt.de}

\author{Tarek Z\"{o}chling}
\address{Technische Universit\"at Darmstadt\\
	Fachbereich Mathematik\\
	Schlossgartenstr.\ 7\\
	64289 Darmstadt\\
	Germany}
\email{zoechling@mathematik.tu-darmstadt.de}

\begin{document}


\thanks{All three  authors acknowledge the support by the DFG project FOR~5528 and thank Franco Flandoli for his insightful discussions on the topic.}

\begin{abstract}
  This article investigates a Sellers-type energy balance model coupled to the primitive equations by a dynamic boundary condition with and without noise on the boundary.
  It is shown that this system is globally strongly well-posed both in the deterministic setting for arbitrary large data in  $\rW^{2(1-\nicefrac{1}{p}),p}$
  for $p \in [2,\infty)$ and in the stochastic setting for arbitrary large data in $\rH^1$. 
\end{abstract}

\maketitle

\begin{center}
   Fachbereich Mathematik, Technische Universit\"{a}t Darmstadt 
\end{center}

\section{Introduction}
\noindent
Consider a two-dimensional energy balance model coupled with the basic equations of geophysical flows, namely, the primitive equations of ocean dynamics. This system integrates an advection-diffusion model for the ocean's interior temperature (subject to a dynamic boundary condition) with the three-dimensional primitive equations. We study two settings: the deterministic case and the stochastic case, in which the stochastic force is modelled by a cylindrical Wiener process \((W_t)_t\) on the boundary. Both settings are of great interest due to the presence of dynamic boundary conditions and the fact that we deal with unique, global, strong solutions.

More specifically, we investigate global, strong well-posedness results for this system, described in the deterministic setting by the set of equations
\begin{equation}
\left\{
\begin{aligned}
\eps \partial_t v^\varepsilon + u^\varepsilon  \cdot \nabla  v ^\varepsilon - \Delta v^\varepsilon  + \nablaH p^\varepsilon  &= 0,  &\quad \text{ in } &\mathcal{O} \times (0,\tau),\\
\dz p^\varepsilon  &= -T,  &\quad \text{ in } &\mathcal{O} \times (0,\tau)\\
\div u ^\varepsilon  &=0, &\quad \text{ in } &\mathcal{O} \times (0,\tau),\\
\partial_t T + u ^\varepsilon  \cdot \nabla T -\Delta T&= 0, \quad &\text{ in }& \mathcal{O} \times (0,\tau), \\
 T|_{\Gamma_u} &= \rho, \quad &\text{ in }&G \times (0,\tau),\\
\partial_t \rho + v ^\varepsilon |_{\Gamma_u} \cdot \nablaH  \rho- \DeltaH \rho + (\dz  T)|_{ \Gamma_u}  &= R(x,\rho), \quad &\text{ in }&G \times (0,\tau),\\
 v^\varepsilon (0)& = v_0, \quad T(0)= T_0.
\end{aligned}
\right.
\label{eq: full model in introduction}
\end{equation}
Here, \(\mathcal{O} = G \times (0,1)\) denotes the spatial domain, where \(G =(0,1)^2\) represents the horizontal domain, and where \(\Gamma_u = G \times \{1\}\) as well as
\(\Gamma_b = G \times \{0\}\) denote the
upper (surface) or lower (bottom) boundary of $\mathcal{O}$, respectively. Furthermore, $u^\varepsilon  = (v^\varepsilon , w^\varepsilon ) : \mathcal{O} \rightarrow \R^3$ denotes the velocity of the ocean and $p^\varepsilon  : \mathcal{O} \rightarrow \R$ denotes ocean's pressure on a climate, slow time-scale (several years), where $\varepsilon \sim \nicefrac{\tau_{\text{fast}}}{\tau_{\text{slow}}}$ characterises the separation between the fast (weather) and slow (climate) time-scales; see \autoref{sec: model formulation} for more details. Moreover,
$T :\mathcal{O} \to \R$ denotes the temperature of the ocean and $\rho =T|_{\Gamma_u}:G \to \R$ the temperature evaluated at the surface. The operators $\DeltaH$ and  $\nablaH$ denote the horizontal Laplacian and
the horizontal gradient, that is $\DeltaH := \partial^2_x + \partial^2_y$ and $\nablaH:=(\partial_x,\partial_y)^T$. Moreover, \(R(x,\rho)\) is a reaction term modelling the net radiative balance at the
surface described in detail below.

In the stochastic setting, the force is given through a cylindrical Wiener process  $(W_t)_t$, acting via a given process $H_\rho$ on the temperature on the boundary. In this case, the system reads
\begin{equation}
\left\{
\begin{aligned}
\eps \d v^\varepsilon + u^\varepsilon \cdot \nabla v^\varepsilon \d t - \Delta v^\varepsilon \d t + \nablaH p ^\varepsilon \d t  &= 0 ,  &\quad \text{ in } &\mathcal{O} \times (0,\tau),\\
\dz p^\varepsilon &= -T,  &\quad \text{ in } &\mathcal{O} \times (0,\tau)\\
\div u ^\varepsilon&=0, &\quad \text{ in } &G \times (0,\tau),\\
\d  T + u ^\varepsilon \cdot \nabla T \d t-\Delta T \d t&= 0, \quad &\text{ in }& \mathcal{O} \times (0,\tau), \\
 T|_{\Gamma_u} &= \rho, \quad &\text{ in }&G \times (0,\tau),\\
\d \rho + \vbar ^\varepsilon \cdot \nablaH  \rho \d t- \DeltaH \rho \d t+ (\dz  T)|_{ \Gamma_u} \d t &= R(x,\rho) \d t + H_\rho \d W, \quad &\text{ in }&G \times (0,\tau),\\
 v^\varepsilon(0)& = v_0, \quad T(0)= T_0,
\end{aligned}
\right.
\label{eq: primitive + EBM simplified stohcastic}
\end{equation}
Note that here the transport term on the boundary is replaced by the transport along the average of the velocity given by $\vbar^\eps = \int_0^1 v^\eps(\cdot,\xi) \d \xi$.

Both systems \eqref{eq: full model in introduction} as well as \eqref{eq: primitive + EBM simplified stohcastic} are supplemented by the boundary conditions
\begin{equation}\label{eq:bc}
\begin{aligned} 
    (\dz v^\varepsilon)|_{\Gamma_u \cup \Gamma_b} = w^\varepsilon|_{\Gamma_u \cup \Gamma_b} = 0, \quad (\dz T)|_{\Gamma_b}=0 \ \text{ and } \ u^\varepsilon,p^\varepsilon,T,\rho \quad \text{ are periodic on }\ \partial G \times [0,1] .
\end{aligned} 
\end{equation}

Note that the domain $G = (0,1)^2$, when endowed with periodic boundary conditions, can be naturally identified with the torus $\T^2$.

Energy balance model serves often as a simplified climate model. For more information on deterministic energy balance models and their applications,  we refer
e.g., to \cite{Ghil1976,Ashwin2020, Cannarsa23, DelSarto24a}.
The first pioneering work on stochastic climate models goes back to K. Hasselmann in \cite{Hasselmann1976}, where noise arises as a result of the interaction of the many
components constituting the climate system, in particular as the effect of the faster component on the slower ones; see \cite{Imkeller2001} for a formulation of Hasselmann's programme
in modern mathematical language. For further information on stochastic energy balance models, we refer to \cite{Fraedrich2001,Diaz2022, DelSarto24}. To account for contributions of both slow and fast processes in the noise term, we may consider $
W = (W^1,W^2), $
where \(W^1\) and \(W^2\) are independent cylindrical Wiener processes, and set $
H_\rho = (H_\rho^1, \sqrt{\varepsilon} H_\rho^2).$ Here, \(H_\rho^1\) captures unresolved variability arising from slow processes, while \(H_\rho^2\), scaled by \(\sqrt{\varepsilon}\), represents the cumulative effect of fast (weather-scale) processes on the slow climate dynamics, according with standard multiscale stochastic modeling approaches; see \cite{DelSarto24a}. 
For simplicity of notation and exposition, in the remainder of this work we avoid explicitly considering the two-component formulation of \(H_\rho\) and treat a single effective noise term.

A central aspect in our system is the  analysis of  the {\em dynamic  boundary condition} for the surface temperature, given by \(T|_{\Gamma_u} = \rho\). Dynamic boundary conditions have, of course, been previously studied, both in the weak and strong sense.  We refer here e.\ g., to the work of Denk, Pr\"uss and Zacher \cite{Denk2008}, who considered dynamic boundary conditions
for diffusion, surface diffusion, the linearised Stefan problem, and linearised free boundary value problems for the Navier-Stokes equations. Furthermore, Pr\"uss, Racke, and Zheng
\cite{Pruss2006} obtained a strong well-posedness result for the Cahn-Hillard equation with dynamic boundary condition.   

Let us also mention that D\'iaz and Tello  \cite{Diaz07} proved the existence of a weak solution to certain energy balance models, where the fluid is {\em passive}, in the sense that the velocity $u$ of the
fluid is prescribed and not part of the problem. 

There is a large literature on stochastic  fluid systems or stochastic heat systems with noise or  with noise on the boundary. We refer here for example to the work of
Cerrai and Freidlin \cite{Cerrai2011} or Debussche, Fuhrmann and Tessitore \cite{Debussche2007}, Brzezniak, Goldys, Peszat and Russo \cite{Russo2015}, Binz, Hieber, Hussein and Saal \cite{Binz2024}. Regarding the primitive equations with additive or multiplicative noise, we refer to \cite{Guo2009} and \cite{Deubussche2012,Vicol2014}.
For further results focusing on linear boundary conditions and solutions with $\rL^2(\rH^1)$ regularity, we refer to  \cite{Bonaccorsi2006,Wang2009}.
Note that in contrast to the above cited  studies, we consider \emph{dynamic boundary conditions in the strong sense subject to noise}.

The aim of this article is threefold: firstly, we couple the energy balance model with an \emph{active} fluid, meaning that the ocean's velocity is not a given function but governed by the
coupling with the incompressible primitive equations. More precisely, the coupling is given by a dynamic boundary condition for the temperature on the boundary subject to a  transport term on the
upper boundary \(\Gamma_u\) of the form  $v\vert_{\Gamma_u} \cdot \nablaH \rho$. Here, \(v\) denotes the horizontal component of the (unknown) velocity \(u = (v,w)\).

Secondly, we aim to prove the  existence of a {\em unique, global, strong } solution to the system \eqref{eq: full model in introduction} for {\em arbitrary large} initial data belonging to
$\rW_\per^{2(1-\nicefrac{1}{p}),p} \times \rW_\per^{2(1-\nicefrac{1}{p}),p}$ for $p \in [2,\infty)\setminus \{3\}$, satisfying the  horizontal solenoidal divergence-free condition. In the special case of $p=2$, we thus
obtain global strong well-posedness of solutions for arbitrary large data in $\rH^1 \times \rH^1$.    
This is rather surprising, since we would be unable to do so when the  geophysical fluid would be replaced, for example, by the Navier-Stokes equations.

Thirdly, concerning  equations \eqref{eq: full model in introduction} and \eqref{eq: primitive + EBM simplified stohcastic}, let us recall that we are dealing with
{\em dynamic boundary conditions in the strong sense subject to noise}. It seems that, independent of the system under consideration, there are no results known concerning strong solutions
for noise on the boundary subject to dynamic boundary conditions. Our final aim is to prove the existence of a unique, global, strong solution 
in the pathwise sense with regularity $\rC(\rH^1) \cap \rL^2(\rH^2)$  for arbitrary large data in $\rH^1 \times \rH^1$.

At this point, some words about our approach are in order. We will comment on the physics behind our model in \autoref{sec:physics} and \autoref{sec: model formulation}. In order to establish global, strong  well-posedness results for the system \eqref{eq: full model in introduction} and \eqref{eq: primitive + EBM simplified stohcastic}, we associate with the temperature equations the differential
operator $\mathcal{A}(D,D_\H)$, defined by
\begin{equation*}
    \mathcal{A}(D,D_\H)(T,\rho) :=  (\Delta T, -\gamma \partial_z T + \DeltaH \rho ),
\end{equation*}
subject to the boundary condition $(\dz T)|_{\Gamma_b}=0$. Here, $\gamma$ denotes the trace operator and $\DeltaH$ the two-dimensional Laplace operator with periodic boundary conditions.
For $p \in (1,\infty)$, consider the $\rLp$-realization of $\mathcal{A}$ in  $\rX^{(T,\rho)}_0 := \rLp(\mathcal{O}) \times \rLp(G)$, defined by
\begin{equation}
    \begin{aligned}\label{eq: A1}
        A(T,\rho) = \mathcal{A}(D,D_\H)(T,\rho) \mbox{ with }
        \D(A) = \bigr  \{ (T,\rho) \in \rH_\per^{2,p}(\mathcal{O}) \times \rH_\per^{2,p}(G) \ \colon T|_{\Gamma_u}= \rho, \  (\dz T)|_{\Gamma_b} =0 \bigl \}.
    \end{aligned}
\end{equation}

We will then show in \autoref{sec: lin and loc} that the operator $A$ admits a bounded $\mathcal{H}^\infty$-calculus on $\rLp(\mathcal{O}) \times \rLp(G)$ with angle $\Phi_A<\pi/2$. 
That is, $A$ is sectorial with spectral angle $\varphi_A<\Phi_A$, and there exists $\Phi_A\in(\varphi_A,\pi/2)$ such that the following holds. 
Let $
   \Sigma_{\Phi_A} := \{ \lambda\in\C\setminus\{0\} : |\arg\lambda|<\Phi_A\}$, $\rho(\lambda):=\frac{\lambda}{(1+\lambda)^2},
$ and
$
   \mathcal{H}^\infty_0(\Sigma_{\Phi_A})
   :=\{ f\in\mathcal{H}^\infty(\Sigma_{\Phi_A}) : 
      \exists\,C,\varepsilon>0  \text{ s.t. } |f(\lambda)|\le C\,|\rho(\lambda)|^\varepsilon,\;\lambda\in\Sigma_{\Phi_A} \}.
$
Then $A$ has a bounded $\mathcal{H}^\infty$-calculus of angle $\Phi_A$ if
$
   \|f(A)\|_{\mathcal{L}(\rLp(\mathcal{O})\times\rLp(G))}
   \le C_{\Phi_A} \sup_{z\in\Sigma_{\Phi_A}}|f(z)|$ for all $f\in\mathcal{H}^\infty_0(\Sigma_{\Phi_A}).
$
This property is crucial because in the case where $\Phi_A \le \pi/2$ it implies maximal $\rL^p$-regularity for $A$, a key tool in our analysis.

This yields, as a first step, the existence of a unique, local, strong solution to system  \eqref{eq: primitive + EBM simplified}. 
Energy estimates are now crucial for extending this solution to a global one for arbitrary large data. Note that the presence of the non-linear boundary transport term
$ v^\eps|_{\Gamma_u} \cdot \nabla_H \rho$ naturally leads to the appearance of the term $ \int_{G} v^\eps|_{\Gamma_u} \cdot \nabla_H \rho \cdot \rho$ in the classical energy estimate used for
the global well-posedness result. This term, however,  does not vanish, as $v^\eps|_{\Gamma_u}$ no longer satisfies the (horizontal) divergence-free condition.
The control of this term thus crucially depends on obtaining a uniform $\rL^\infty$ bound for $\rho$. We derive such a bound by a suitable  maximum principle, hereby
  adapting a result of Ewald and Temam \cite{Temam2001} to our setting. This  enables us to
close the energy estimates and to prove global well-posedness for arbitrary large data belonging to $\rW_\per^{2(1-\nicefrac{1}{p}),p} \times \rW_\per^{2(1-\nicefrac{1}{p}),p}$, for $p \in [2,\infty) \setminus \{3\}$.

To investigate the effect of noise in the temperature equation on the boundary, we treat the fluid and temperature equations separately, and first consider the temperature as a
stochastic evolution equation. 
The boundedness of the  $\mathcal{H}^\infty$-calculus of $A$ on $\rLp(\mathcal{O}) \times \rLp(G)$ has now another important consequence for the noise on  the boundary.  
In fact, the linearised problem enjoys stochastic maximal regularity, see \cite{Veraar2012a, Agresti2022a, Agresti2022b}. 
To handle the nonlinear problem, we employ the Da Prato-Debussche trick (\cite{DPD02, DPD03}). The central idea is to split the solution into two components: a rough part that captures the full irregularity of the boundary noise via the stochastic convolution of the linearised system, and a smoother remainder that satisfies a deterministic evolution equation with random coefficients. This transformation has two advantages: first, it isolates the irregular behaviour of the noise in an explicitly controlled component; second, it reduces the original stochastic PDE with nonlinearities into a random PDE where classical energy estimates can be applied. In this way, the nonlinear terms act only on the more regular remainder, and we can establish the existence of unique, global, strong solutions in the pathwise sense.

Regarding open problems, let us note that our assumption of spatial periodicity for the horizontal domain $G$ is a simplification. The most realistic domain would require more complex geometries (e.\ g., spherical caps or annuli) and considerably more technical effort. The present work should thus be viewed as a first analytical step in the simplified periodic setting of the torus $\T^2$, providing a foundation for future studies on more realistic domains.

This article is organized as follows. In \autoref{sec: model formulation}, we present the formulation of our 2D-EBM coupled to an active 3D ocean, described by the primitive equations.
In \autoref{sec: prelim + main}, we provide an equivalent formulation of the model along with the fundamental elements of our functional setting. We also  state there our main results regarding
the existence and uniqueness of a global strong solution as well as  additional regularity properties for the deterministic setting. In \autoref{sec: lin and loc}, we prove the local well-posedness by establishing a
bounded $\mathcal{H}^\infty$-calculus for the operator associated with  the temperature and its  dynamic boundary condition. Next, in \autoref{sec: global}, we establish the
global well-posedness of the strong solution to \eqref{eq: full model in introduction}.  Our  proof is based on a maximum principle for the ocean's temperature
(obtained under the additional hypothesis of an \(\rL^\infty\)- bound on the initial temperature), the maximal $L^p$-regularity estimate for the linearized system and energy estimates. Finally, in \autoref{sec: stoch}, we provide a proof of the global existence result concerning the stochastic system \eqref{eq: primitive + EBM simplified stohcastic}.

\section{Description behind existing energy balance models}\label{sec:physics}

Modelling the Earth's climate system is, naturally, a highly challenging task. Scientists therefore rely on a hierarchy of climate models, ranging from elementary elementary averaged models that elegantly describe the main drivers of Earth's climate to fully coupled general circulation models that incorporate fluid dynamics as well as
chemical and biological processes. Among these, energy balance models , though not coupled to a fluid, occupy a fundamental role. These were were introduced independently by Budyko and Sellers in 1969 \cite{Budyko1969, Sellers1969}.
They assume that the Earth's surface temperature \( T \) evolves according to the balance between the radiation absorbed by the planet and the radiation it emits. Ultimately,
while EBMs simplify many of the intricate processes governing climate, they remain indispensable for their clarity, analytical tractability, and capacity to reveal the
underlying structure of Earth's climate dynamics. For more information we refer to  the articles \cite{Ashwin2020, Cannarsa23, DelSarto24a, DelSarto24}.

Let us now describe various energy balance models in more detail. We denote by  \(\Gamma_u \subset \mathbb{R}^2\) an open set representing a simplified description of Earth's surface.
A two-dimensional energy balance model (2D-EBM) is given by the following equation
$$
c(x) \, \partial_t \rho = \operatorname{div}\bigl( \kappa(x) \nabla \rho \bigr) + R_a(x,\rho) - R_e(x,\rho), \quad \text{in } \Gamma_u \times (0,\tau),
$$
where \(\rho(x,t)\) denotes the Earth's surface temperature at the point \(x \in \mathcal{O}\) and time \(t \in (0,\tau)\), \(c(x)\) represents the heat capacity, and
\(\operatorname{div}\bigl( \kappa(x) \nabla \rho \bigr)\) models the heat diffusion in a simplified way. For a more detailed discussion on this topic, see \cite{North17}.
The functions \(R_a\) and \(R_e\) denote the radiation absorbed and emitted by the planet, respectively. The parametrizations of these terms vary slightly among different authors;
we discuss the details in \autoref{sec: model formulation}. A one-dimensional energy balance model (1D-EBM) is obtained by averaging the 2D-EBM along the longitude direction,
resulting in a temperature \(\overline{\rho}\) that depends only on the latitude, see  \cite{Ghil1976}.

In 1990, Watts and Morantine (\cite{Watts1990}) proposed a two-dimensional (latitude-depth) upwelling-diffusion ocean model coupled with a 1D-EBM. Their main goal was to study the
interaction between Earth's surface and the deep ocean, particularly starting from the analysis rapid climate change recorded paleoclimatic data. Starting from the work of Watts and Morantine,
D\'iaz and Tello (\cite{Diaz07}) investigated a 2D-EBM coupled to a 3D ocean model, considering, however,  a {\em passive} field \(u\colon \mathcal{O} \to \mathbb{R}^3\) for the ocean's velocity.
The domain is simplified to $
\mathcal{O} = G \times (0,1) \subset \mathbb{R}^3,$
with \(G \subset \mathbb{R}^2\) representing the horizontal component of the ocean. Denoting by \(\Gamma_u = G \times \{1\}\) the upper boundary (i.\ e., the ocean's surface), the model
proposed by D\'iaz and Tello takes the form:
\begin{equation}
\left\{
\begin{aligned}
\partial_t T + u \cdot \nabla T - \Delta T &= 0,  &\quad \text{in } &\mathcal{O} \times (0,\tau),\\[1mm]
T\vert_{\Gamma_u} &= \rho,  &\quad \text{in } &G \times (0,\tau),\\[1mm]
\partial_t \rho - \divH\bigl( \kappa(x) \nablaH \rho \bigr) + \bigl( \partial_z T \bigr)\vert_{G \times \{0\}} + F(x, \nablaH \rho) &= R(x, \rho),  &\quad \text{in } &G \times (0,\tau),\\[1mm]
\overline{T}(0) &= \overline{T}_0,
\end{aligned}
\right.
\label{eq: model Diaz and Tello}
\end{equation}
where \(F\colon \Gamma_u \times \mathbb{R}^2 \to \mathbb{R}\) represents, in a simplified manner, the transport along the boundary.
The main improvements introduced by D\'iaz and Tello compared to previous models are: the inclusion of an additional spatial dimension in the EBM on the boundary and to the ocean's
interior temperature, the ocean model now incorporates not only upwelling (as in Watts and Morantine's model) but also horizontal transport via the term \(u \cdot \nabla T\), 
and the addition of a simplified boundary transport term \(F(x, \nablaH \rho)\).

\section{The new model: coupling to a geophysical flow}\label{sec: model formulation}

In this section, we describe the coupling of the energy balance model to the primitive equations, which are a basic model in geophysical fluid dynamics for oceanic and atmospheric dynamics.
We assume that the velocity of the fluid is governed by the primitive equations. These are derived from the 
Navier-Stokes equations under the assumption of hydrostatic balance for the pressure. We refer to works of Majda \cite{Majda2003} and Vallis \cite{Vallis2017} for details.

The primitive equations  have been introduced in a series of
papers  by Lions, Temam and Wang ~\cite{LTW:92a, LTW:92b, LTW:95}. In their pioneering work, they proved the existence of a weak solution to the
primitive equations; its uniqueness remains an open problem until today. Since then, both the compressible and incompressible primitive equations have been the subject of
intensive mathematical investigations. A celebrated result of Cao and Titi \cite{CT:07} shows that the incompressible primitive equations, subject to Neumann boundary conditions, are
globally strongly well-posed in the three-dimensional setting  for arbitrarily large data in $\rH^1$. For different approaches and results, including those in
critical spaces, we refer e.\ g., to \cite{KZ:07a,HK:16,GGHHK:20b}.  

Consider a cylindrical domain $ \mathcal{O} = G \times (0,1),$ with $G = (0,1)^2$. This domain models the ocean, with the upper boundary \(\Gamma_u = G \times \{1\}\) representing the ocean's surface. The surface temperature, denoted by $
\rho = T\vert_{G \times \{1\}},$
is governed by the two-dimensional energy balance model (2D-EBM)
\begin{equation}\label{eq:EBM}
\partial_t \rho - \DeltaH \rho + v\vert_{\Gamma_u} \cdot \nablaH \rho + (\partial_z T)\vert_{\Gamma_u} = R(x,\rho), \quad \text{in } G \times (0,\tau).
\end{equation}
In the above, all physical constants (such as the heat capacity and diffusion constants) have been normalized to one. In \eqref{eq:EBM}, the term \(\DeltaH \rho\) models the horizontal diffusion of temperature, \(v\vert_{\Gamma_u}\cdot \nablaH \rho\) captures the horizontal advection at the surface, and \((\partial_z T)\vert_{\Gamma_u}\) represents the vertical heat flux from the ocean interior to the surface. The reaction term is given by
\begin{equation}\label{eq: radiation}
R(x, \rho) = Q(x)\,\beta(\rho) - |\rho|^3 \rho, \quad x \in \Gamma_u,\; \rho \in \mathbb{R},
\end{equation}
where the solar radiation \(Q \in C^1_b(G)\) is positive. The outgoing radiation is modelled via the Stefan-Boltzmann law. The co-albedo function $\beta$ is parametrized as
\begin{equation}\label{eq: coalbedo}
\beta(\rho) = \beta_1 + (\beta_2 - \beta_1)\, \frac{1+\tanh(\rho-\rho_{\mathrm{ref}})}{2},
\end{equation}
with \(0<\beta_1<\beta_2\) corresponding to the co-albedo values for ice-covered and ice-free conditions, respectively. The parameter \(\rho_{\mathrm{ref}}\) denotes a critical reference temperature, usually chosen as \(\rho_{\mathrm{ref}} = - 263 K\), at which ice turns white (see \cite{Diaz2022}). We remark that our EBM is of Sellers-type, i.e. with Lipschitz continuous co-albedo. The other class of EBM the Budyko type, characterised discontinuous co-albedo. For the latter, it is known that even for the EBM alon, uniqueness of the solution does not hod; see, for instance, \cite{Diaz97} for a discussion of the non-uniquenss of weak solutions.

The temperature \(T\colon \mathcal{O} \to \mathbb{R}\) in interior of the ocean is coupled to the surface model through the dynamic boundary condition at \(\Gamma_u\). We assume that the temperature $T$ satisfies a Heat equation and is transported along the full velocity field and the surface energy balance \eqref{eq:EBM} acts as a dynamic boundary condition for the temperature \(T\) at \(\Gamma_u\). Hence we arrive at the system 
\begin{equation}
\left\{
\begin{aligned}
\partial_t T + u \cdot \nabla T -\Delta T&= 0, \quad &\text{ in }& \mathcal{O} \times (0,\tau), \\
 T|_{\Gamma_u} &= \rho, \quad &\text{ in }&G \times (0,\tau),\\
\partial_t \rho + v|_{\Gamma_u} \cdot \nablaH  \rho- \DeltaH \rho + (\dz  T)|_{ \Gamma_u}  &= R(x,\rho), \quad &\text{ in }&G \times (0,\tau),\\
{T}(0) &= {T}_0.
\end{aligned}
\right.
\label{eq: model Diaz and Tello1}
\end{equation}

Our aim is to study an \emph{active} fluid, meaning that the ocean velocity $u$ is not prescribed but instead determined by the \emph{incompressible primitive equations}, which are governed by the system
\begin{equation}\label{eq:primitive}
\left\{
\begin{aligned}
\partial_t v + u \cdot \nabla v - \Delta v + \nablaH p &= 0,  &\quad \text{in } \mathcal{O} \times (0,\tau),\\
\partial_z p &= -T,  &\quad \text{in } \mathcal{O} \times (0,\tau),\\
\div u &= 0,  &\quad \text{in } \mathcal{O} \times (0,\tau),\\
v(0) &= v_0. 
\end{aligned}
\right.
\end{equation}
Here $u = (v,w)$ where $v \colon \mathcal{O}\to \R^2$ denotes the horizontal part of the velocity field and $w\colon \mathcal{O}\to \R$ denotes its vertical part. Moreover, $p \colon \mathcal{O}\to \R$ denotes the pressure of the ocean. We emphasize that the primitive equations of the ocean \eqref{eq:primitive} 
describe the \emph{fast} fluid motions, whereas the temperature model 
\eqref{eq: model Diaz and Tello1} evolves on a \emph{slow} time scale.

To couple equations \eqref{eq: model Diaz and Tello1} and \eqref{eq:primitive} togheter, it is natural to separate the 
characteristic time scales and apply a nondimensionalization. Let $\tau'$ denote the fast (weather/seasonal) time, and introduce the slow time variable
\[
   \tau = \varepsilon \tau', 
   \qquad 0 < \varepsilon \ll 1,
\]
where $   \varepsilon \sim \nicefrac{\tau_{\text{fast}}}{\tau_{\text{slow}}}
   \ \text{(e.\ g., days--months over years--decades).}$
  We then define the rescaled velocity and pressure fields of the ocean by
\[
   u^\varepsilon(t',x,y,z) := u(\varepsilon t,x,y,z), 
   \qquad 
   p^\varepsilon(t',x,y,z) := p(\varepsilon t,x,y,z),
\]
and consider the combined systems, consisting of the rescaled primitive equations and the climate model \eqref{eq: model Diaz and Tello1}, given by the equations \eqref{eq: full model in introduction} for the deterministic model, and equations \eqref{eq: primitive + EBM simplified stohcastic} for the stochastic model, where the transport term is replaced by the transport along the averaged velocity field.

Summarizing, the system \eqref{eq: full model in introduction} combines a two-dimensional reaction-diffusion model on the surface with a three-dimensional fluid dynamic model for the ocean interior.
The 2D-EBM captures the primary mechanisms of radiative balance through the parametrization in \eqref{eq: radiation}--\eqref{eq: coalbedo}, while the primitive equations account for the
full dynamical behaviour of the ocean. In particular, the advective term \(v^\eps \vert_{\Gamma_u} \cdot \nablaH \rho\) emphasizes the role of horizontal transport in redistributing heat along the surface.

By treating the 2D-EBM as a dynamic boundary condition, the model naturally couples the surface radiative processes with the interior dynamics. This coupling is fundamental in
understanding how surface phenomena (such as albedo feedbacks) influence the overall oceanic temperature distribution.

Lastly, we introduce an additive noise term to our model on the dynamic boundary condition to capture the unresolved variability arising from fast weather processes
interacting with slower climate dynamics. Physically, this noise represents random fluctuations in the boundary heat flux, for example, variations in solar radiation or
atmospheric conditions, that directly affect the ocean's surface temperature. Such fluctuations are critical in accurately modelling the stochastic nature of the climate system,
as originally suggested by Hasselmann \cite{Hasselmann1976}.

From an analytical point of view, incorporating noise into the dynamic boundary condition, as done in \eqref{eq: primitive + EBM simplified stohcastic}, adds a new additional difficulty.
In the deterministic setting, analytical tools such as
the maximum principle allow us to obtain energy estimates. However, these techniques do not extend to the stochastic framework. We thus modify the transport term on the
boundary in our stochastic model by replacing the velocity on the boundary by its mean velocity. This allows us to close the energy estimates and to prove existence and uniqueness of
global, strong solutions for arbitrary large data in $\rH^1$ also for the stochastic model in the pathwise sense.

\section{Preliminaries and Main Results}\label{sec: prelim + main}
Let $\tau>0$ be fixed. In this section we present our main results regarding the unique, strong and global solvability of the deterministic equations \eqref{eq: full model in introduction} as well as 
of the stochastic model \eqref{eq: primitive + EBM simplified stohcastic}. Both systems are supplemented with the dynamic boundary conditions \eqref{eq:bc}. For simplicity of the notation we denote in the following the rescaled oceanic velocity $u^\eps$ and the rescaled pressure $p^\eps$ by $u$ and $p$ respectively. Moreover, as it does not affect the analysis, we set the scaling constant $\eps \equiv 1$. 

\subsection{The deterministic model}  \mbox{}\\
Consider the system \eqref{eq: full model in introduction} supplemented with the boundary conditions \eqref{eq:bc}. The equation of state $\dz p = -T$ yields
\begin{equation*}
    p(t,x,y,z) = p_s(t,x,y) -  \int_0^z T(\cdot, \xi) \d \xi,
\end{equation*}
where $p_s$ denotes the surface pressure, i.\ e., $p_s(t,x,y)=p(t,x,y,z=0)$. Moreover, the vertical velocity is determined by the divergence-free condition, namely, 
\begin{equation*}
    w(t,x,y,z) = -\int_0^z \divH v(t,x,y,\xi) \d \xi.
\end{equation*}
Using this, as well as the boundary conditions for $w$, we rewrite \eqref{eq: full model in introduction} as
\begin{equation}
\left\{
\begin{aligned}
\partial_t v + v \cdot \nablaH  v+ w(v)\cdot \dz v - \Delta v + \nablaH p_s &= \nablaH \int_0^z T(\cdot,\xi) \d \xi,  &\quad \text{ in } &\mathcal{O} \times (0,\tau),\\
\divH \vbar &=0, &\quad \text{ in } &G \times (0,\tau),\\
\partial_t T + u \cdot \nabla T -\Delta T&= 0, \quad &\text{ in }& \mathcal{O} \times (0,\tau), \\
 T|_{\Gamma_u} &= \rho, \quad &\text{ in }&G \times (0,\tau),\\
\partial_t \rho + v|_{\Gamma_u} \cdot \nablaH  \rho- \DeltaH \rho + (\dz  T)|_{ \Gamma_u}  &= R(x,\rho), \quad &\text{ in }&G \times (0,\tau),\\
 v(0)& = v_0, \quad T(0)= T_0,
\end{aligned}
\right.
\label{eq: primitive + EBM simplified}
\end{equation}
supplemented by the boundary conditions \eqref{eq:bc}.

\subsection{The stochastic model} \mbox{}\\
Consider the system \eqref{eq: primitive + EBM simplified stohcastic} supplemented with the boundary conditions \eqref{eq:bc}.
Let $(\Omega, \mathcal{F},P)$ be a probability space with a filtration $\mathcal{F}=(\mathcal{F}_t)_t$. An 
$\mathcal{F}$-cylindrical Brownian motion on a Hilbert space $\mathcal{U}$ is a bounded linear operator $W: \rL^2((0,\infty);\mathcal{U})
\to \rL^2(\Omega)$ such that for all $f,g\in \mathcal{U}, t'\geq t \geq 0$:
\begin{enumerate}
\item[(i)] The random variable $W(t)f:=\mathcal{W}(\mathds{1}_{[0,t]}\otimes f)$ is centred Gaussian and $\mathcal{F}_t$-measurable.
\item[(ii)] $\E[W(t')f\cdot W(t)g]= t \left< f,g \right>_{\mathcal{U}}$.
\item[(iii)] The random variable $W(t')f-W(t)f$ is independent of $\mathcal{F}_t$.
\end{enumerate}
If $\mathcal{U}$ is separable and $(e_n)_n$ an orthonormal basis of $\mathcal{U}$, then $\beta_n(t):=W(t)e_n$ is a standard $\mathcal{F}$-Brownian motion, and we have the representation 
\begin{align*}
W(t)f = \sum_{n=1}^\infty \beta_n(t) \left< f,e_n \right>_{\mathcal{U}}.
\end{align*}
Hence, $W(t): \mathcal{U} \to L^2(\Omega)$, $W(t)= \sum_{n=1}^\infty \beta_n(t) \left< \cdot ,e_n \right>_{\mathcal{U}}$ defines a family of linear operators.

To deal with the noise in the temperature equation we treat the fluid equation and the temperature equation separately and rewrite the equations concerning the temperature as a stochastic evolution equation of the form
\begin{equation}
    \label{eq: temp stoch}
    \d \begin{pmatrix}
        T \\ \rho
    \end{pmatrix} -A\begin{pmatrix}
        T \\ \rho
    \end{pmatrix}\d t = \begin{pmatrix} -u \cdot \nabla T \\ -\vbar \cdot \nablaH \rho + R(x,\rho)        
    \end{pmatrix} \d t + \begin{pmatrix}
        0 \\ H_\rho
    \end{pmatrix} \d W,
\end{equation}
where the operator $A$ is introduced in \eqref{eq: A1}. Secondly, we investigate the linearized system 
\begin{equation}\label{eq: stoch eq}
    \d \begin{pmatrix}
        Z_T \\ Z_\rho
    \end{pmatrix} -A \begin{pmatrix}
        Z_T \\ Z_\rho
    \end{pmatrix} \d t = \begin{pmatrix}
        0 \\ H_\rho
    \end{pmatrix} \d W,
\end{equation}
by means of stochastic maximal regularity due to van Neerven, Veraar and Weiss \cite{Veraar2012a}. This is possible since $-A+ \omega$ admits a bounded $\mathcal{H}^\infty$-calculus on $\rX_0 = \rLp(\mathcal{O}) \times \rLp(G)$
of angle strictly less than $\pi/2$ for $\omega>0$ sufficiently large. This will be established in \autoref{sec: lin and loc}. Defining the remainder term $ (\tilde{T},\tilde{\rho})^\top := (T,\rho)^\top - (0,Z_\rho)^\top$ we see that $(\tilde{T},\tilde{\rho})$ satisfies the equations
\begin{equation}
    \label{eq: deterministic stoch}
    \dt \begin{pmatrix}
        \tilde{T} \\ \tilde{\rho}
    \end{pmatrix}  -A \begin{pmatrix}
         \tilde{T} \\ \tilde{\rho}
    \end{pmatrix} = \begin{pmatrix}
        -u \cdot \nabla \tilde{T} \\ -\vbar \cdot \nablaH (\tilde{\rho} + Z_\rho) + R(x,\tilde{\rho} + Z_\rho) 
    \end{pmatrix},
\end{equation}
which can be regarded as a deterministic evolution equation by the maximal regularity properties of $Z = (0,Z_\rho)$. By defining $(T,\rho)^\top = (\tilde{T},\tilde{\rho})^\top+(0,Z_\rho)^\top$, we conclude that in order to solve \eqref{eq: primitive + EBM simplified stohcastic} it is sufficient to consider the deterministic system 
\begin{equation}
\left\{
\begin{aligned}
\partial_t v + v \cdot \nablaH  v+ w(v)\cdot \dz v - \Delta v + \nablaH p_s &= \nablaH \int_0^z \tilde{T}(\cdot,\xi) \d \xi,  &\quad \text{ in } &\mathcal{O} \times (0,\tau),\\
\divH \vbar &=0, &\quad \text{ in } &G \times (0,\tau),\\
\partial_t \tilde{T} + u \cdot \nabla \tilde{T}  -\Delta \tilde{T}&= 0, \quad &\text{ in }& \mathcal{O} \times (0,\tau), \\
 \tilde{T}|_{\Gamma_u} &= \tilde{\rho}, \quad &\text{ in }&G \times (0,\tau),\\
\partial_t \tilde{\rho} +\vbar \cdot \nablaH (  \tilde{\rho}+Z_\rho)- \DeltaH \tilde{\rho} + (\dz  \tilde{T})|_{ \Gamma_u}  &= R(x,\tilde{\rho}+Z_\rho), \quad &\text{ in }&G \times (0,\tau),\\
 v(0)& = v_0, \quad \tilde{T}(0)= \tilde{T}_0,
\end{aligned}
\right.
\label{eq: primitive + EBM simplified1}
\end{equation}
subject to the boundary conditions \eqref{eq:bc}, where $Z=(0,Z_\rho)$ is the given solution of the linear stochastic evolution equation \eqref{eq: stoch eq}.

\noindent

In conclusion, our approach outlined above to solving the stochastic model relies on the application of the Da Prato-Debussche trick. Specifically, to study the existence of a unique, strong, global solution to \eqref{eq: primitive + EBM simplified stohcastic} in the pathwise sense, we proceed as follows:
(a) we consider the semilinear auxiliary stochastic evolution equation \eqref{eq: temp stoch} for temperature;
(b) we linearize \eqref{eq: temp stoch} to obtain \eqref{eq: stoch eq}, whose solution enjoys stochastic maximal regularity;
(c) we then consider, pathwise, the remainder between the solutions of \eqref{eq: temp stoch} and \eqref{eq: stoch eq}, which satisfies the deterministic, non-autonomous evolution equation \eqref{eq: primitive + EBM simplified1}. The solution to this final equation is then obtained using, in part, the techniques adopted for the deterministic model \eqref{eq: primitive + EBM simplified}.

\subsection{Functional setting}\mbox{}\\
Let us introduce several  function spaces, which are used in the sequel. Given  $s \in \R$ and $p,q \in (1,\infty)$, we
denote by $\rW^{s,q}(\mathcal{O})$ the fractional Sobolev spaces and by $\rH^{s,q}(\mathcal{O})$ the Bessel potential spaces, where $\mathcal{O} \subset \R^3$ is an open set.
As usual, we set $\rH^{0,q}(\mathcal{O}) := \rL^q(\mathcal{O})$, $\rH^s(\mathcal{O}):=\rH^{s,2}(\mathcal{O})$ and note that $\rH^{s,q}(\mathcal{O})$ coincides with the 
classical Sobolev space $\rW^{m,q}(\mathcal{O})$ whenever $s=m \in \N$. For more information on these function spaces we refer e.g. to \cite{Ama:19}. 

We  also introduce the terminology needed to describe periodic boundary conditions on $\Gamma_l = \partial G \times [0,1]$ as well as on $\partial G$. Given $s \in [0,\infty)$ and $p,q \in (1,\infty)$ we define the spaces
\begin{equation*}
    \rH^{s,q}_\per (\mathcal{O}) := \bar{\rC^\infty_\per (\bar{\mathcal{O}})}^{\| \cdot \|_{\rH^{s,q}(\mathcal{O})}} \ \text{ and } \ \rH^{s,q}_\per (G) := \bar{\rC^\infty_\per (\bar{G})}^{\| \cdot \|_{\rH^{s,q}(G)}},
\end{equation*}
where horizontal periodicity is modelled by the function spaces $\rC^\infty_\per(\bar{\mathcal{O}})$ and $\rC^\infty_\per(\bar{G})$ defined in \cite{HK:16}. Of course, we interpret $\rH^{0,q}_\per$ as $\rLq$. 
For more information regarding spaces equipped with periodic boundary conditions, see also \cite{HK:16, GGHHK:20b}. Next, we define the space of \emph{hydrostatically solenoidal vector fields} by
\begin{equation*}
    \rL^q_{\sigmabar}(\mathcal{O}) = \overline{\{v \in \rC^\infty(\overline{\mathcal{O}};\R^2) \colon \divH \vbar = 0 \}}^{\| \cdot \|_{\rL^q(\mathcal{O})}}.
\end{equation*}
Their role in the analysis of the primitive equations parallels that of the solenoidal vector fields for the Navier-Stokes equations.

\subsection{Main results}\mbox{}\\
We are now in position to formulate our main results concerning the local well-posedness as well as the global well-posedness of strong solution for arbitrary large data, both for the deterministic model \eqref{eq: primitive + EBM simplified} and for the model \eqref{eq: primitive + EBM simplified1}, associated, pathwise, to the stochastic model \eqref{eq: primitive + EBM simplified stohcastic}. Our requirements on the set of initial data are collected in the following assumption.

\vspace{.2cm}
\noindent
{\bf Assumption (A)}: Let $p \in [2,\infty) \setminus \{ 3 \}$ and assume that 
 \begin{enumerate}[(i)]
        \item $  v_0 \in \rW_\per^{2(1-\nicefrac{1}{p}),p}( \mathcal{O};\R^2)   \cap \rLp_{\sigmabar} (\mathcal{O};\R^2) \ \text{ satisfies }\ (\dz v_0)|_{\Gamma_u \cup \Gamma_b} =0 \ \text{ if } \ p>3$.  
        \item $T_0 \in \rW_\per^{2(1-\nicefrac{1}{p}),p}( \mathcal{O}) \ $ satisfies $\ T_0|_{\Gamma_u} \in \rW_\per^{2(1-\nicefrac{1}{p}),p}( G) \ $ and
          $\ (\dz T_0)|_{\Gamma_b}=0 \ \text{ if } \ p>3$.
    \end{enumerate}

\begin{thm}[Local and Global Well-Posedness of \eqref{eq: primitive + EBM simplified}]\label{thm: local wp} \mbox{} \\
Let $\tau >0$ and assume that $(v_0,T_0)$ satisfy assumption (A).\vspace{.1cm} \\ 
$\mathrm{(a)}$  There exists $0<a=a(v_0,T_0)\leq \tau$ such that the system \eqref{eq: primitive + EBM simplified} subject to the boundary conditions \eqref{eq:bc} admits a unique, strong solution $(v, T)$ satisfying 
\begin{equation*}
    \begin{aligned}
         v &\in \rH^{1,p}(0,a;\rLp_{\sigmabar}(\mathcal{O};\R^2)) \cap \rLp(0,a;\rH_{\per}^{2,p}(\mathcal{O};\R^2) \cap \rLp_{\sigmabar}(\mathcal{O};\R^2)) \ \text{ and } \\
         T &\in \rH^{1,p}(0,a;\rLp(\mathcal{O})) \cap \rLp(0,a;\rH_\per^{2,p}(\mathcal{O})) \ \text{ such that }\ T|_{\Gamma_u} \in \rH^{1,p}(0,a;\rLp(G)) \cap \rLp(0,a;\rH^{2,p}_\per(G)).
    \end{aligned}
\end{equation*}
$\mathrm{(b)}$ Assume additionally that $T_0 \in \rL^\infty(\mathcal{O})$ such that $T_0|_{\Gamma_u} \in \rL^\infty(G)$. Then the solution exists globally in time, i.\ e., $a =\tau$.
\end{thm}

The solution constructed above enjoys additional regularity, formulated in the following corollary.

\begin{cor}[Regularity]\label{cor:regularity} \mbox{} \\
    Let $p \geq 2$ and $(v,T,\rho)$ the unique, global strong solution to \eqref{eq: primitive + EBM simplified} subject to \eqref{eq:bc} given in \autoref{thm: local wp}. Then  
    \begin{equation*}
        v \in \rC^\alpha((0,\tau) \times \mathcal{O}), \ T \in \rC^\alpha((0,\tau) \times \mathcal{O}) \ \text{ and } \ \rho \in \rC^\alpha((0,\tau) \times G),
    \end{equation*}
    where $\alpha \in \{ \infty, \omega\}$.
\end{cor}

\begin{thm}[Global well-posedness of \eqref{eq: primitive + EBM simplified1}] \mbox{} \\
    \label{thm: global stoch}
    Let $0<\tau<\infty$, $H=(0,H_\rho) \in \rL^2(\Omega\times (0,\tau);\rL^2(\mathcal{U}; \rH_\per^1(\mathcal{O}) \times \rH_\per^1(G)))$ and assume
    \begin{equation*}
        v_0 \in \rH_\per^1(\mathcal{O};\R^2) \cap \rL^2_{\sigmabar}(\mathcal{O};\R^2) \ \text{ as well as }\ \tilde{T}_0 \in \rH_\per^1(\mathcal{O}) \ \text{ such that } \ \tilde{T}_0|_{\Gamma_u} \in \rH_\per^1(G).
    \end{equation*}
    Then the system \eqref{eq: primitive + EBM simplified1} subject to the boundary conditions \eqref{eq:bc} admits a unique, global, strong solution $(v, \tilde{T})$ satisfying 
\begin{equation*}
    \begin{aligned}
         v &\in \rH^{1}(0,\tau;\rL^2_{\sigmabar}(\mathcal{O};\R^2)) \cap \rL^2(0,\tau;\rH_{\per}^{2}(\mathcal{O};\R^2) \cap \rL^2_{\sigmabar}(\mathcal{O};\R^2)) \ \text{ and } \\
         \tilde{T} &\in \rH^{1}(0,\tau;\rL^2(\mathcal{O})) \cap \rL^2(0,\tau;\rH_\per^{2}(\mathcal{O})) \ \text{ such that }\ \tilde{T}|_{\Gamma_u} \in \rH^{1}(0,\tau;\rL^2(G)) \cap \rL^2(0,\tau;\rH^{2}_\per(G)).
    \end{aligned}
\end{equation*}
Then the solution $(v,T,\rho)$ of \eqref{eq: primitive + EBM simplified stohcastic} is given by $(v,T,\rho)^\top=(v,\tilde{T},\tilde{T}|_{\Gamma_u})^\top+(0,Z)^\top$, where $Z=(0,Z_\rho)^\top$ is the solution of the linear stochastic evolution equation \eqref{eq: stoch eq}.
\end{thm}

\begin{rem}
For simplicity of exposition, we state the results for the models \eqref{eq: primitive + EBM simplified} and \eqref{eq: primitive + EBM simplified1} with $\varepsilon=1$, and we drop the $\varepsilon$-notation on $u$ and $p$. All results, however, remain valid for any fixed choice of $\varepsilon>0$ and imply global well-posedness for the original models \eqref{eq: full model in introduction}, \eqref{eq: primitive + EBM simplified stohcastic}.

\end{rem}

\begin{rem}
  {\rm The surface pressure $p_s$ is omitted in the above theorems as it can be recovered from \eqref{eq: primitive + EBM simplified}$_1$ as well as \eqref{eq: primitive + EBM simplified1}$_1$ by averaging and applying the horizontal divergence to the
  equation, see also \cite{Hieber2016}.} 
\end{rem}

\section{Linear theory and local well-posedness: deterministic setting}\label{sec: lin and loc}

In this section, we establish the maximal regularity results concerning the decoupled linear equations associated with \eqref{eq: primitive + EBM simplified} and prove local well-posedness of the coupled
system. Let $0<\tau \leq \infty$  given a well as $\omega_1, \omega_2 \in \R$.  The linear system for the velocity is given by
\begin{equation}
\left\{
\begin{aligned}
\partial_t v  - \Delta v + \omega_1 v + \nablaH p_s &= f_1,  &\quad \text{ in } &\mathcal{O} \times (0,\tau),\\
\divH \vbar &=0, &\quad \text{ in } &G \times (0,\tau),\\
(\dz v) |_{\Gamma_u \cup \Gamma_b} &= 0, \quad &\text{ in }&G \times (0,\tau),\\
& = v_0, \quad & \text{ in }& \mathcal{O},
\end{aligned}
\right.
\label{eq: lin primitive}
\end{equation}
and the linear equations for the temperature are given by
\begin{equation}
\left\{
\begin{aligned}
\partial_t T -\Delta T +\omega_2 T&= f_2, \quad &\text{ in }& \mathcal{O} \times (0,\tau) \\
 T|_{\Gamma_u} &= \rho, \quad &\text{ in }&G \times (0,\tau),\\
\partial_t \rho - \DeltaH \rho +\omega_2 \rho + (\dz  T)|_{ \Gamma_u}  &= f_3, \quad &\text{ in }&G \times (0,\tau),\\
(\dz T)|_{\Gamma_b} &= 0, \quad &\text{ in }&G \times (0,\tau),\\
T(0)& = T_0, \quad & \text{ in }& G.
\end{aligned}
\right.
\label{eq: simple temperature model linear}
\end{equation}
In order to state the well-posedness result for the linearized primitive equations \eqref{eq: lin primitive}, we define the time trace space
\begin{equation*}
    \rX^v_\gamma  \coloneqq \left \{
    \begin{aligned}
      \{ &v_0 \in\rW_\per^{2(1-\nicefrac{1}{p}),p}(\mathcal{O};\R^2) \cap \rLp_{\sigmabar} (\mathcal{O};\R^2)\ \text{ and }\ (\partial_z v_0)|_{\Gamma_u \cup \Gamma_b}=0 \}, \quad &&p>3, \\
    &\rW_\per^{2(1-\nicefrac{1}{p}),p}(\mathcal{O};\R^2) \cap \rLp_{\sigmabar} (\mathcal{O};\R^2), \quad &&1<p<3, \ p \neq \nicefrac{3}{2}.
    \end{aligned}
 	\right.
\end{equation*}
We then recall the following result from \cite{GGHHK:20b}.

\begin{lem}
    \label{lem: lin prim}
     Let $p\in (1,\infty)\setminus \{ \frac{3}{2}, 3 \}$. Then there exists $\omega_1\geq 0$ such that the system \eqref{eq: lin primitive} admits a unique, strong solution $v$ and $p_s$ satisfying 
     \begin{equation*}
         v \in \E^v_{1,\infty} := \rH^{1,p}(\R_+;\rLp_{\sigmabar}(\mathcal{O};\R^2) \cap \rLp(\R_+;\rH_{\per}^{2,p}(\mathcal{O};\R^2) \cap \rLp_{\sigmabar}(\mathcal{O};\R^2)) \ \text{ and }\ p_s \in \rLp(\R_+;\dot \rH_\per^{1,p}(G))
     \end{equation*}
     if and only if the data $(f_1,v_0)$ satisfies 
     \begin{equation*}
        f_1 \in \rLp((0,\infty)\times \mathcal{O})=:\E_{0,\infty}^v  \ \text{  and } \ v_0 \in \rX^v_\gamma.
     \end{equation*}
     In that case, there exists a constant $C>0$ such that $v$ and $p_s$ satisfy the estimate 
     \begin{equation*}
         \| v \|_{\E^v_{1,\infty}} + \| \nablaH p_s \|_{\rLp(\R_+ \times G)} \leq C \bigr ( \| f_1 \|_{\E^v_{0,\infty}} + \| v_0 \|_{\rX^v_\gamma} \bigl ).
     \end{equation*}
\end{lem}

To establish our well-posedness result for \eqref{eq: primitive + EBM simplified}, we associate with its left-hand side the differential operator $\mathcal{A}(D,D_\H)$ defined by
\begin{equation*}
    \mathcal{A}(D,D_\H)(T,\rho) :=     
    (\Delta T, -\gamma \partial_z T + \DeltaH \rho ),
\end{equation*}
subject to the boundary condition $(\dz T)|_{\Gamma_b}=0$. Here, $\gamma$ denotes the trace operator and $\DeltaH$ denotes the two-dimensional Laplace operator with periodic boundary conditions. For $p \in (1,\infty),$ consider the $\rLp$-realization of $\mathcal{A}$ in the ground space $\rX^{(T,\rho)}_0 := \rLp(\mathcal{O}) \times \rLp(G)$, which is defined by
\begin{equation}
    \begin{aligned}\label{eq: A}
        A(T,\rho) &:= \mathcal{A}(D,D_\H)(T,\rho),\\
        \D(A) &:= \bigr  \{ (T,\rho) \in \rH_\per^{2,p}(\mathcal{O}) \times \rH_\per^{2,p}(G) \ \colon T|_{\Gamma_u}= \rho, \  (\dz T)|_{\Gamma_b} =0 \bigl \}.
    \end{aligned}
\end{equation}
It was shown in  \cite{Denk2008} that the $\rLp$-realization $A_{1-\nicefrac{1}{p}}$ of $\mathcal{A}$ in the ground space $\rX^{(T,\rho)}_{1-\nicefrac{1}{p},0} = \rLp(\mathcal{O}) \times \rW^{1-\nicefrac{1}{p},p}(G)$ has the property of maximal $\rLp$-regularity. In the following,  we show that the operator $A$ admits a bounded $\mathcal{H}^\infty$-calculus of angle strictly less than $\pi/2$. A similar result holds
true when considering the $\rLp$-realization of $\mathcal{A}$ in the ground space $\rX_{s,0}^{(T,\rho)} := \rLp(\mathcal{O}) \times \rW_\per^{s,p}(G)$ for $s\in [-\frac{1}{p}, 1-\frac{1}{p}]$, see \autoref{rem: Hunendlich}.

\begin{lem}
    \label{lem: max reg A}
     Let $p\in (1,\infty)$. Then there exists $\omega_0 \in \R$ such that for all $\omega>\omega_0$ the operator $-A+\omega$ admits a bounded $\mathcal{H}^\infty$-calculus on $\rX^{(T,\rho)}_0 = \rLp(\mathcal{O}) \times \rLp(G)$ of angle $\Phi_A < \pi/2$.
\end{lem}
\begin{proof}
  The proof is based on a similarity transform of the operator $A$ and the fact that the bounded $\mathcal{H}^\infty$-calculus is preserved under similarity,
  see e.g.,  \cite{DHP03}. We divide the proof into two steps. 
  \begin{step} The aim of the first step is to construct an invertible mapping $S\colon \rX_0^{(T,\rho)} \to  \rX_0^{(T,\rho)}$ such that the operator $G:= S^{-1}AS$ has diagonal domain. To this end,
    consider the auxiliary problem
    \begin{equation}
\left\{
\begin{aligned}
\Delta \theta &= 0,  &\quad \text{ in } &\mathcal{O},\\
\theta|_{\Gamma_u} &= \phi, \quad &\text{ in }&G,\\
(\dz \theta) |_{\Gamma_b} &= 0, \quad &\text{ in }&G,\\
\end{aligned}
\right.
\label{eq: aux}
\end{equation}
with $\phi \in \rLp(G)$. We denote its solution operator by $\rL_0 \colon \rLp(G) \rightarrow \rLp(\mathcal{O})$ and show that $\rL_0$ is well-defined and bounded. Notice that for $\phi \in \rC^\infty_\per(G)$
there exists a unique solution $\theta \in \rH^{2,p}_\per(\mathcal{O})$ of \eqref{eq: aux} and as a consequence $\rL_0$ exists as a densely defined operator in $\rLp(G)$ with domain $\rC^\infty_\per(G)$. This allows us to define a unique adjoint $\rL_0'\colon \rL^{p'}(\mathcal{O}) \to \rL^{p'}(\mathcal{O})$ with $\nicefrac{1}{p}+ \nicefrac{1}{p'}=1$. To show that $\rL_0$ is well-defined and bounded, we prove that $\rL_0'$ can be extended to a bounded operator from $\rL^{p'}(G)$ to $\rL^{p'}(\mathcal{O})$. For this, consider $\phi \in \rC^\infty_\per(G)$ and $k \in \rL^{p'}(\mathcal{O})$, and set $g := \Delta^{-1}k$, where $\Delta^{-1}$ denotes the Laplace operator with homogeneous boundary conditions on $\rL^{p'}(\mathcal{O})$. Then, by Green's second identity and horizontal periodicity, we obtain
\begin{equation*}
    \langle \rL_0 \phi , k \rangle_{\rL^2(\mathcal{O})} = \langle \rL_0 \phi , \Delta g \rangle_{\rL^2(\mathcal{O})} = \langle \Delta \rL_0 \phi , g \rangle_{\rL^2(\mathcal{O})} + \int_G \phi(x)(\dz g)(x, 1) \d x = \langle \phi , \gamma \dz \Delta^{-1} k \rangle_{\rL^2(G)},
\end{equation*}
since $\Delta \rL_0 \phi =0$. We conclude that $\rL_0' = \gamma \dz \Delta^{-1}$, where $\gamma \dz$ denotes the distributional derivative in normal direction with maximal domain $\D(\gamma \dz ) = \{ f \in \rL^{p'}(\mathcal{O}) \colon \gamma \dz f \in \rL^{p'}(G)\}$. Note that $\gamma \dz$ is relatively $(-\DeltaH)^{\delta}$-bounded with $\delta > \nicefrac{1}{2}+ \nicefrac{1}{2p}$, since by trace theory it holds true that $\rH^{r,p}_\per(\mathcal{O}) \subset \D(\gamma \dz)$, for $r > 1 + \nicefrac{1}{q}$.

By density of $\rC^\infty_\per(G)$ in $\rL^{p'}(G)$, we extend the operator $\rL_0'$ to a bounded operator from $\rL^{p'}(G) \to \rL^{p'}(\mathcal{O})$, which implies that $\rL_0$ is bounded
from $\rLp(G)$ to $\rLp(\mathcal{O})$. Moreover, the unique solution of $\eqref{eq: aux}$ satisfies the a-priori estimate $\| \theta \|_{\rLp(\mathcal{O})} \leq C \| \phi \|_{\rLp(G)}$ for some constant $C>0$. Defining 
\begin{equation*}
    S: = \begin{pmatrix}
        \mathrm{Id} & -\rL_0 \\ 0 & \mathrm{Id}
    \end{pmatrix} \colon \rX_0^{(T,\rho)} \to \rX_0^{(T,\rho)}, \ \text{ with inverse }\ S^{-1} = \begin{pmatrix}
        \mathrm{Id} & \rL_0 \\ 0 & \mathrm{Id}
    \end{pmatrix},
\end{equation*}
we consider the operator matrix $G$ defined by
\begin{equation*}
    G := S^{-1}A S = \begin{pmatrix}
        \Delta_0 & -\rL_0 ( \DeltaH +\rN) \\
        -\gamma \dz & \DeltaH + \rN
    \end{pmatrix},
\end{equation*}
with diagonal domain $\D(G) = (T,\rho) \in \{ \rH_\per^{2,p}(\mathcal{O}) \times \rH_\per^{2,p}(G) \colon T|_{\Gamma_u}=0, \ (\dz T)_{\Gamma_b}=0 \}$. Here, $\Delta_0$ denotes the Laplace operator on $\rLp(\mathcal{O})$ with homogeneous mixed boundary conditions and $\rN$ denotes the \emph{Dirichlet-to-Neumann operator} defined by 
\begin{equation*}
    \rN\phi := \gamma \dz \rL_0 \phi, \ \D(\rN):= \{ \phi \in \rLq(G) \colon \rL_0 \phi \in \D(\gamma\dz) \}.
\end{equation*}
Observe that $\rH^{s,p}_\per(G) \subset \D(\rN)$ for all $s>1$, since $\rL_0 \rH_\per^{s,p}(G) \subset \D(\gamma \dz)$ for $s>0$. Hence, the Dirichlet-to-Neumann operator $\rN$ is relatively $(-\DeltaH)^\delta$-bounded for $\delta >\nicefrac{1}{2}$. 
\end{step}
\begin{step}
The second step is devoted to showing that the transformed operator matrix $G$ admits a bounded $\mathcal{H}^\infty$-calculus on  $\rX_0^{(T,\rho)}$ of angle $\Phi_G <\pi/2$. By similarity, we then conclude the same result for the operator matrix $A$.
To this end, we split the matrix $G$ as
\begin{equation*}
    G = G_0 +Q, \ \text{ with }\ G_0 =  \begin{pmatrix}
        \Delta & -\rL_0 \DeltaH \\
        0 & \DeltaH 
    \end{pmatrix} \text{ and }\ Q= \begin{pmatrix}
        0 & -\rL_0 \rN \\
        -\gamma \dz & \rN 
    \end{pmatrix}.
\end{equation*}
To prove the bounded $\mathcal{H}^\infty$-calculus of $G_0$, we verify in the following the assumptions of  \cite[Theorem~3.2]{MR2047641}. First, by \cite{MR2047641} and \cite{MR3019280}, it follows that, up to a shift, the operators $-\Delta$ on $\rLp(\mathcal{O})$ and $-\DeltaH$ on $\rLq(G)$ possess a bounded $\mathcal{H}^\infty$-calculus. Hence, we obtain $\D(\Delta^\delta)= \rH_\per^{2\delta,p}(\mathcal{O})$ and $\D((\DeltaH)^{1+\delta})= \rH_\per^{2+2\delta,p}(G)$ for $\delta >0$. Moreover, by the considerations in the previous step it follows that
\begin{equation*}
    -\rL_0 \DeltaH (  \rH_\per^{2+2\delta,p}(G)) =-\rL_0 \rH_\per^{2\delta,p}(G) \subset \rH_\per^{2\delta,p}(\mathcal{O}) = \D(\Delta^\delta).
\end{equation*}
The closed graph theorem implies that $\| \Delta^\delta(\rL_0 \DeltaH)y \|_{\rLp(\mathcal{O})} \leq C \| (\DeltaH)^{1+\delta}y \|_{\rLp(G)}$ for all $y \in \rH_\per^{2+2\delta,p}(G)$. We conclude that, up to a shift, the operator matrix $G_0$ admits a bounded $\mathcal{H}^\infty$-calculus of angle strictly less that $\pi/2$.

Finally, we treat $G$ as a perturbation of $G_0$ and the results in the previous step yield that the normal derivative $-\gamma \dz$ as well as the Dirichlet-to-Neumann operator $\rN$ are relatively $(-\DeltaH)^\delta$-bounded for a suitable $\delta >0$. By the properties of the Dirichlet map $\rL_0$, it  follows that $\rL_0 \rN$ is relatively $(-\Delta)^\delta$-bounded for all $\delta > \nicefrac{1}{2}$. The claim then follows from perturbation theory for the $\mathcal{H}^\infty$-calculus, see for instance \cite{MR2047641}.
\end{step}
\end{proof}

\begin{rem}\label{rem: Hunendlich}
    {\rm The above result can be extended in the following way. Consider the $\rLp$-realization $A_s$ of $\mathcal{A}$ in the ground space $\rX_{s,0}^{(T,\rho)} = \rLp(\mathcal{O}) \times \rW_\per^{s,p}(G)$ for $s\in [-\frac{1}{p}, 1-\frac{1}{p}]$. Then there exists $\omega_0 \in \R$ such that for all $\omega>\omega_0$ the operator $-A_s+\omega$ admits a bounded $\mathcal{H}^\infty$-calculus on $\rX^{(T,\rho)}_{s,0}$ of angle $\Phi_{A_s} < \pi/2$. The proof is analogous to the one presented above.}
\end{rem}

An immediate consequence of \autoref{lem: max reg A} is the solvability of \eqref{eq: simple temperature model linear}, provided that the forces $(f_2,f_3)$ lie in appropriate spaces. We readily verify that the trace space $\rX^{(T,\rho)}_\gamma = (\rX^{(T,\rho)}_0, \D(A))_{1-\nicefrac{1}{p}} $ is characterized by 
\begin{equation*}
		\rX^{(T,\rho)}_\gamma \coloneqq \left \{
    \begin{aligned}
      \{ &T_0 \in\rW_\per^{2(1-\nicefrac{1}{p}),p}(\mathcal{O})\ \colon T_0|_{\Gamma_u} \in \rW_\per^{2(1-\nicefrac{1}{p}),p}(G) \ \text{ and }\ (\dz T_0)|_{\Gamma_b}=0 \}, \quad &&p>3, \\
   \{ &T_0 \in\rW_\per^{2(1-\nicefrac{1}{p}),p}(\mathcal{O}) \ \colon T_0|_{\Gamma_u} \in \rW_\per^{2(1-\nicefrac{1}{p}),p}(G)  \}, \quad &&1<p<3, \ p \neq \nicefrac{3}{2}.
    \end{aligned}
 	\right.
\end{equation*}
Then the maximal $\rLp$-regularity of the operator $A$ can be stated as follows.
\begin{cor}
    \label{cor: max reg A}
    Let $p\in (1,\infty)\setminus \{ \frac{3}{2}, 3 \}$. Then there exists $\omega \geq 0$ such that the problem \eqref{eq: primitive + EBM simplified} admits a unique, strong solution $T$ satisfying 
    \begin{equation*}
        \begin{aligned}
            (T,\rho) \in \E^{(T,\rho)}_{1,\infty} := \rH^{1,p}(\R_+; \rX^{(T,\rho)}_0) \cap \rLp(\R_+;\D(A))
        \end{aligned}
    \end{equation*}
    if and only if the data satisfies 
    \begin{equation*}
        (f_2,f_3) \in \rLp(0, \tau;\rX^{(T,\rho)}_0) :=\E^{(T,\rho)}_{0,\infty}\ \text{ and }\ T_0 \in \rX_\gamma^T.
    \end{equation*}
In that case, there exists a constant~$C > 0$ such that $T$ satisfies the estimate
    \begin{equation}
        \label{eq: max reg estimate}
        \| (T ,\rho)\|_{\E^{(T,\rho)}_{1,\infty}}  \leq C \bigl ( \| (f_2,f_3) \|_{\E^{(T,\rho)}_{0,\infty}}  + \| T_0 \|_{\rX^{(T,\rho)}_\gamma} \bigr ).
    \end{equation}
\end{cor}

Next, to show local well-posedness of \eqref{eq: primitive + EBM simplified} we estimate the non-linear terms in the respective data spaces. For clarity, we define the solution spaces for $T$ and $T|_{\Gamma_u}$ separately, namely,
\begin{equation}\label{eq: max reg spaces}
    \begin{aligned}
        \E^T_{1,\tau} &:= \rH^{1,p}(0,\tau;\rLp(\mathcal{O})) \cap \rLp(0,\tau;\rH_\per^{2,p}(\mathcal{O}))\  \text{ and }\\ \E^\rho_{1,\tau}&:= \rH^{1,p}(0,\tau;\rLp(G)) \cap \rLp(0,\tau;\rH_\per^{2,p}(G)).
    \end{aligned}
\end{equation}

\begin{lem}
    \label{lem: esimates nonlinearities version 2}
    Let $\tau'<1$ and $p \geq 2$. Then there exists $\tau \in (0,\tau']$ as well as a constant $~C>0$, depending only on $p$ and $\tau'$, such that
    \begin{enumerate}
        \item[(i)] $\ \| R(x,\rho) \|_{\rLp((0,\tau) \times G)} \leq C \tau ^{\nicefrac{1}{p}}+ C \| \rho\|^4_{\E_{1,\tau}^\rho}$,
        \item[(ii)] $\| v|_{\Gamma_u} \cdot \nablaH \rho \|_{\rLp((0,\tau) \times G)} \leq C \| v \|_{\E^v_{1,\tau}} \cdot  \| \rho\|_{\E_{1,\tau}^\rho}$ ,
        \item[(iii)] $ \| u \cdot \nabla T \|_{\rLp((0,\tau) \times \mathcal{O})} \leq C \| v \|_{\E^v_{1,\tau}} \cdot \| T \|_{\E^T_{1,\tau}}\ $ and $ \ \|\nablaH \int_0^z T(\cdot,\xi) \d \xi \|_{\rLp((0,\tau)\times \mathcal{O})} \leq C \tau^{\nicefrac{1}{2p}} \| T \|_{\E^T_{1,\tau}}$,
        \item[(iv)] $ \| u \cdot \nabla v \|_{\rLp((0,\tau) \times \mathcal{O})} \leq C \| v \|_{  \E^v_{1,\tau}}^2$,
    \end{enumerate}
    whenever $v\in \E^v_{1,\tau}$, $T \in \E^T_{1,\tau}$, $\rho \in \E^\rho_{1,\tau}$, and $R$ is given as in \eqref{eq: radiation} for $Q \in \rC^1( G)$.
\end{lem}
\begin{proof}
(i)    From \eqref{eq: radiation} we deduce that
    \begin{equation*}
        \| R(x,\rho) \|_{\rLp((0,\tau) \times G)} \leq \| Q(x)\beta(\rho) \|_{\rLp((0,\tau) \times G)}+ \| \rho^4 \|_{\rLp((0,\tau) \times G)}.
    \end{equation*}
   H\"older's inequality yields
    $$
     \| Q(x)\beta(\rho) \|_{\rLp((0,\tau) \times G)} \leq C \|Q \|_{\rL^{\infty}(G)} \cdot  \| \beta(\rho) \|_{\rLp((0,\tau) \times G)},
    $$
    where we used that $Q \in \rL^\infty(G)$. Furthermore, by the parametrisation \eqref{eq: coalbedo}, we get
    $$
    \| \beta(\rho) \|_{\rLp((0,\tau) \times G)} \leq C \bigl ( \tau^{\nicefrac{1}{p}} +  \| \tanh(\rho- \rho_{\mathrm{ref}}) \|_{\rLp((0,\tau) \times G)} \bigr  ) \leq C\tau^{\nicefrac{1}{p}},
    $$
    since $|\tanh(x)|\leq 1$ for all $x \in \R$. Next, the mixed derivative theorem and choosing $\theta \in (0,1)$ appropriately imply that $\E^\rho_{1,\tau}\hookrightarrow \rH^{\theta,p}(0,\tau;\rH_\per^{2(1-\theta),p}(G)) \hookrightarrow \rL^{4p}(0,\tau;\rL^{4p}(G))$. Therefore, by H\"older's inequality, 
    \begin{equation*}
        \| \rho^4 \|_{\rLp((0,\tau) \times G)} \leq \| \rho \|^4_{\rL^{4p}(0,\tau;\rL^{4p}(G))} \leq C \| \rho \|^4_{\E^\rho_{1,\tau}}.
    \end{equation*}

(ii) In view of H\"older's inequality, the trace theorem and Sobolev embedding we obtain
\begin{equation*}
    \begin{aligned}
        \| v|_{\Gamma_u} \cdot \nablaH \rho \|_{\rLp((0,\tau) \times G)} &\leq C \| v \|_{\rL^{2p}(0,\tau;\rW^{\nicefrac{1}{2p},2p}(\mathcal{O}))} \cdot \| \rho\|_{\rL^{2p}(0,\tau;\rW^{1,2p}(G))} \\ &\leq C \| v \|_{\rL^{2p}(0,\tau;\rW^{\nicefrac{3}{2p},p}(\mathcal{O}))} \cdot \| \rho\|_{\rL^{2p}(0,\tau;\rW^{1+\nicefrac{1}{p},p}(G))}.
    \end{aligned}
\end{equation*}
Then, the mixed derivative theorem yields 
\begin{equation*}
    \E_{1,\tau}^v \hookrightarrow \rL^{2p}(0,\tau;\rW_\per^{\nicefrac{3}{2p},p,p}(\mathcal{O})), \text{ for }p \geq 1 \ \text{ and } \ \E^\rho_{1,\tau} \hookrightarrow \rL^{2p}(0,\tau;\rH_\per^{1 +\nicefrac{1}{p},p}(G)), \text{ for }p\geq 2.
\end{equation*}
Hence,
\begin{equation*}
     \| v|_{\Gamma_u} \cdot \nablaH \rho \|_{\rLp((0,\tau) \times G)} \leq C \|v \|_{\E^v_{1,\tau}} \cdot \| \rho\|_{\E^\rho_{1,\tau}}.
\end{equation*}

(iii) The first estimate can be deduced as in \cite[Lemma~5.1]{Hieber2016}. Concerning the second term, Jensen's inequality as well as interpolation yields 
\begin{equation*}
     \|\nablaH \int_0^z T(\cdot,\xi) \d \xi \|_{\rLp((0,\tau)\times \mathcal{O})} \leq C \| T \|_{\rLp(0,\tau;\rH^{1,p}(\mathcal{O}))} \leq C \| T \|^{\nicefrac{1}{2}}_{\rLp((0,\tau)\times \mathcal{O})}  \| T \|^{\nicefrac{1}{2}}_{\rLp(0,\tau;\rH^{2,p}( \mathcal{O}))} \leq C \tau^{\nicefrac{1}{2p}} \| T \|_{\E^T_{1,\tau}} 
\end{equation*}

(iv) According to \cite[Lemma~5.1]{HK:16}, we have
$$
\| u \cdot \nabla v \|_{\rLp( \mathcal{O})} \leq C \| v \|^2_{\rH^{1+ \nicefrac{1}{p},p}(\mathcal{O})}.
$$
Taking the $\rL^p$-norm in time and taking advantage of the embedding $\E_{1,\tau}^v \hookrightarrow \rL^{2p}(0,\tau;\rH^{1+\nicefrac{1}{p},p}(\mathcal{O}))$, for $p\geq 2$ , we conclude that 
$$
 \| u \cdot \nabla v \|_{\rLp((0,\tau) \times \mathcal{O})}  \leq C \|v\|^2_{\rL^{2p}(0,\tau;\rH^{1+\nicefrac{1}{p},p}(\mathcal{O}))}\leq C \| v \|_{  \E^v_{1,\tau}}^2.
$$
\end{proof}

\begin{proof}[Proof of \autoref{thm: local wp}(a)]
Let us remark first that in the sequel we implicitly consider the shifted linearized equations in order to exploit \autoref{lem: lin prim} and \autoref{cor: max reg A}. Since the resulting linear term added on the right-hand side can be estimated easily, we do not explicitly mention this shift in the remainder of the proof. Next, observe that \autoref{lem: lin prim} and \autoref{cor: max reg A} guarantee the existence of a reference solution $(v_*, T_*, \rho_*)$ to \eqref{eq: lin primitive} and \eqref{eq: simple temperature model linear} subject to homogeneous right-hand sides and initial data $(v_0,T_0)$. We then define the set 
    \begin{equation*}
        \begin{aligned}
                \mathbb{B}_r:=  \{ (v, T, \rho) \in \E_{1, \tau}  : (v, T, \rho) - (v_*, T_*, \rho_*) \in \prescript{}{0}{\E_{1, \tau}} \  \text{ and } \  \| (v, T, \rho) - (v_*, T_*, \rho_*) \|_ {\E_{1, \tau}} \leq r  \},
        \end{aligned}
    \end{equation*}
    where the prescript $\prescript{}{0}{}$ denotes the space with homogeneous initial conditions. For $(v_1, T_1, \rho_1) \in \mathbb{B}_r,$ we denote by 
    \begin{equation*}
        \Psi (v_1, T_1, \rho_1):=  (v, T, \rho) \in \E_{1,\tau} 
    \end{equation*}
    the solution to the linearized system \eqref{eq: primitive + EBM simplified} where the respective right-hand sides are given by 
    \begin{equation*}
        \nablaH \int_0^z T_1(\cdot, \xi) \d \xi - v_1 \cdot \nablaH v_1 - w(v_1) \cdot \dz v_1, \ -v_1 \cdot \nablaH T_1 - w(v_1) \cdot \dz T_1 \ \text{ and } \ R(x,\rho)- v_1|_{\Gamma_u} \cdot \nablaH \rho.
    \end{equation*}
    Then, by \autoref{lem: lin prim}, \autoref{cor: max reg A} and \autoref{lem: esimates nonlinearities version 2}, the mapping $\Psi$ is well defined on $\mathbb{B}_r$. It remains to show that $\Psi$ maps $\mathbb{B}_r$ into itself and that it is a contraction. Specifically, \autoref{lem: lin prim}, \autoref{cor: max reg A} and \autoref{lem: esimates nonlinearities version 2} imply that 
    \begin{equation*}
        \begin{aligned}
            &\| (v, T, \rho) - (v_*, T_*, \rho_*|) \|_ {\E_{1, \tau}} \\ &\leq C \bigl ( \| \nablaH \int_0^z T_1(\cdot, \xi) \d \xi - v_1 \cdot \nablaH v_1 - w(v_1) \cdot \dz v_1 \|_{\E_{0,\tau}^v} \\& \quad + \| (-v_1 \cdot \nablaH T_1 - w(v_1) \cdot \dz T_1,  R(x,\rho_1)- v_1|_{\Gamma_u} \cdot \nablaH \rho_1) \|_{\E_{0,\tau}^T} \bigr ) \\
            & \leq C \bigl ( \tau^{\nicefrac{1}{2p}}\| T_1 \|_{\E_{1,\tau}^T} + \| v_1 \|^2_{\E^v_{1,\tau}} + \| T_1 \|_{\E_{1,\tau}^T} \cdot  \| v_1 \|_{\E_{1,\tau}^v}+ \tau^{\nicefrac{1}{p}} +\| \rho_1 \|_{\E_{1,\tau}^\rho}^4 + \| v_1 \|_{\E_{1,\tau}^v} \| \rho_1 \|_{\E_{1,\tau}^\rho}   \bigr ) \\
            & \leq C \bigl [   \tau^{\nicefrac{1}{2p}} \bigl (r+ \| T_* \|_{\E_{1,\tau}^T} \bigr) +r^2 + \| v_* \|^2_{\E^v_{1,\tau}} + \|T_*\|_{\E^T_{1,\tau}} \cdot \| v_*\|_{\E^v_{1,\tau}}+ \tau^{\nicefrac{1}{p}} + r^4 + \| \rho_* \|^4_{\E_{1,\tau}^\rho} + \| v_* \|_{\E_{1,\tau}^v}\cdot \| \rho_*\|_{\E^\rho_{1,\tau}}  \bigr )
        \end{aligned}
    \end{equation*}
    Since $\| T_* \|_{\E_{1,\tau}^\rho},$ $\| v_* \|_{\E_{1,\tau}^v}$, $\| \rho_*\|_{\E^\rho_{1,\tau}} \to 0$ as $\tau \to 0$, the corresponding terms can be made arbitrarily small by choosing $\tau>0$ sufficiently small. Hence, for appropriately $\tau$ and $r$, the solution map $\Psi$ is a self-map on $\mathbb{B}_r$. The contraction property follows in the same way.

\end{proof}

\section{Global well-posedness and higher regularity: deterministic setting}\label{sec: global}
Throughout this section, assume $0<\tau< \infty$, $p=2$ and let $(v, T, \rho)$ be a solution to \eqref{eq: primitive + EBM simplified} satisfying
\begin{equation}\label{eq:Assumptionsapriori}
    (v, T, \rho) \in \E_{1,\tau} = \E^v_{1,\tau} \times \E^{(T,\rho)}_{1,\tau} = \E^v_{1,\tau} \times \E^T_{1,\tau} \times \E^\rho_{1,\tau},
\end{equation}
with initial data
\begin{equation}\label{eq:Assumptionsapriori2}
    v_0 \in \rH_\per^{1,2}(\mathcal{O}) \cap \rL^2_{\sigmabar}(\mathcal{O}) \ \text{ as well as } \ T_0  \in \rH_\per^{1,2}(\mathcal{O}) \cap \rL^\infty(\mathcal{O})\ \text{ such that }\ T_0|_{\Gamma_u} \in \rH_\per^{1,2}(G) \cap \rL^\infty(G) .
\end{equation} 
A central component of our analysis is the derivation of uniform temperature bounds via the maximum principle. In the following, we adjust an approach  by Ewald and Temam \cite{Temam2001} to the given
situation. These bounds are established independently of the velocity field. Indeed, the non-linear boundary transport term $ v|_{\Gamma_u} \cdot \nabla_H \rho$ introduces an extra term $ \int_{G} v|_{\Gamma_u} \cdot \nabla_H \rho \cdot \rho$ into the energy estimates required to achieve global well-posedness. Since $v|_{\Gamma_u}$ does not meet the horizontal divergence-free condition, this extra term does not cancel out. Keeping it under control is essential, and we do so by obtaining a uniform $\rL^\infty$ bound for $\rho$. In short, the maximum principle is crucial because it gives us the bounds needed to complete the energy estimates and ultimately extend the solution globally.

\begin{prop}[$\rL^\infty_t \rL^\infty_{\mathrm{x}}$-Estimates for $T$ and $\rho$] \label{prop:linftylinfty} \mbox{}\\
   Let $(v,T,\rho)$ and $(v_0,T_0)$ be as in \eqref{eq:Assumptionsapriori} and \eqref{eq:Assumptionsapriori2}. Then there exists a constant $C>0$, independent of time, such that
    $$
    \sup_{t \in [0,\tau]}\| T(t) \|_{\rL^\infty(\mathcal{O})} + \sup_{t \in [0,\tau]} \| \rho(t) \|_{\rL^\infty(G)}  \leq C. 
    $$
\end{prop}
\begin{rem}
    { \rm The constant in \autoref{prop:linftylinfty} is given by $C=\max \{ \| T_0 \|_{\rL^\infty(\mathcal{O})}, \| T_0|_{\Gamma_u}  \|_{\rL^\infty(G)}, \beta^{\nicefrac{1}{4}}_2 \}$.}
\end{rem}
\begin{proof}

We consider the solution of $(T,\rho= T|_{\Gamma_u}) \in \E_{1,\tau}^T$ of the problems
  \begin{equation}
\left\{
\begin{aligned}
\partial_t T +  u \cdot \nabla T- \Delta T &= 0, \quad &\text{ in }& \mathcal{O} \times (0,\tau), \\
T|_{ \Gamma_u}&=  \rho, \quad &\text{ in }&G \times (0,\tau),\\
(\dz T)|_{\Gamma_b} &=  0, \quad &\text{ in }& G  \times (0,\tau),\\
T(0)& = T_0, \quad & \text{ in }& \mathcal{O} ,
\end{aligned}
\right.
\label{eq: equation temperature T maximum principle}
\end{equation}
and
  \begin{equation}
\left\{
\begin{aligned}
\partial_t \rho + v|_{\Gamma_u} \cdot \nablaH \rho+ (\partial_z T)|_{ \Gamma_u}- \DeltaH \rho &= R(x,\rho), \quad &\text{ in }& G \times (0,\tau), \\
\rho(0)& = T_0|_{\Gamma_u}, \quad & \text{ in }& G,\\
\end{aligned}
\right.
\label{eq: equation temperature  rho maximum principle}
\end{equation}
with $u=(v,w)$ and $v \in \E^v_{1,\tau}$. Let $t_1 \in (0,\tau)$ and consider the solution $(T,\rho)$ of \eqref{eq: equation temperature T maximum principle}-\eqref{eq: equation temperature  rho maximum principle} from time $t= 0$ to time $t = t_1$ with global maximum $T_{\mathrm{max}}$ at some point $(x_0,t_0) \in   \mathcal{O} \times [0,t_1]$. We distinguish between the following cases. 
\begin{enumerate}[(a)]
    \item Assume that $(x_0,t_0)$ is at time $t_0=0$. In this case, it is immediate to check that 
    \begin{equation*}
        T_{\mathrm{max}}  \leq  \max \{ \| T_0 \|_{\rL^\infty(\mathcal{O})}, \| T_0|_{\Gamma_u}  \|_{\rL^\infty(G)} \} .
    \end{equation*}
    \item Assume that $(x_0,t_0)\in \Gamma_b \times (0,t_1]$. Then $(\dz T)|_{(t,x)=(x_0,t_0)} \geq 0$ and it follows from the boundary condition  \eqref{eq: equation temperature T maximum principle}$_3$ that $T_{\mathrm{max}} \leq 0$. Arguing analogously for the minimum value, we conclude that $T=0$. 
    \item Assume that $(x_0,t_0) \in  \mathcal{O} \times (0,t_1]$. Consider the rescaled solution $\Tilde{T} = \mathrm{e}^{-\lambda t} T$, for $\lambda>0$. Then $\Tilde{T}$ satisfies a shifted version of \eqref{eq: equation temperature T maximum principle} and it holds true that
$$
\partial_t \Tilde{T}|_{ (x,t)=(x_0,t_0)}\geq 0, \quad   \nabla \Tilde{T}|_{  ((x,t)=(x_0,t_0)} = 0, \quad \Delta \Tilde{T}|_{(x,t)=(x_0,t_0)} \leq 0.
$$
It follows that $\lambda T_{\mathrm{max}} \leq 0$ and thus that $T=0$. 
\item Lastly, we assume that $(x_0,t_0) \in \Gamma_u \times (0,t_1]$. Since $(x_0,t_0)$ is a maximum point of $T$ it also has to be a maximum point of $\rho$ and we see that
\begin{equation*}
    \begin{aligned}
        \partial_t \rho|_{ (x,t)=(x_0,t_0)} \geq 0, \quad (\partial_z T)|_{(x,t)=(x_0,t_0)} \geq 0, \quad  \nablaH \rho|_{(x,t)=(x_0,t_0)} = 0, \quad \DeltaH \rho|_{(x,t)=(x_0,t_0)} \leq 0.
    \end{aligned}
\end{equation*}
In particular, we conclude that 
\begin{equation*}
    T_{\mathrm{max}}^4 \leq Q(x)\beta(T_{\mathrm{max}})
\end{equation*}
and by the representation \eqref{eq: coalbedo} it follows that $T^4_{\mathrm{max}} \leq \beta_2$. A similar argument concerning the minimum value yields $T_{\mathrm{min}}^4 \geq \beta_1$.
\end{enumerate}

\end{proof}

We are now in position to prove that the energy of the system \eqref{eq: primitive + EBM simplified} can be bounded by a continuous function.

\begin{lem}[Energy estimates for $(v,T,\rho)$]\label{lem:energyestimates}  \mbox{}\\
Let $(v,T,\rho)$ and $(v_0,T_0)$ be as in \eqref{eq:Assumptionsapriori} and \eqref{eq:Assumptionsapriori2}. Then there exists a continuous function $B_1$, depending on $\|v_0\|_{\rH^1(\mathcal{O})}$, $\| T_0 \|_{\rH^1(\mathcal{O}) \cap \rL^\infty(\mathcal{O})}$, $\| T_0|_{\Gamma_u}  \|_{\rH^1(G) \cap \rL^\infty(G)} $ and $t$, such that
\begin{equation*}
\|v \|^2_{2} + \| T \|^2_2 + \| \rho\|^2_2 +  \int_0^t \| \nabla v \|^2_{2} + \| \nabla T \|_2^2 + \| \nablaH \rho\|^2_2 +\| \rho \|^5_5 \d s \leq B_1(t),\ \text{ } t \in [0,\tau].
\end{equation*}
\end{lem}
\begin{proof}
    Multiplying \eqref{eq: primitive + EBM simplified}$_1$ by $v$, \eqref{eq: primitive + EBM simplified}$_3$ by $T$, integrating over $\mathcal{O}$ and $G$ respectively, and adding the resulting equations yields, in view of \cite[Lemma~6.3]{HK:16} 
\begin{equation*}
    \begin{aligned}
        &\frac{1}{2}\dt \bigl ( \|v \|^2_2 + \| T \|^2_2 + \| \rho \|^2_2 \bigr) + \| \nabla v \|^2_2 + \| \nabla T \|^2_2 + \|\nablaH \rho \|^2_2 + \| T|_{\Gamma_b} \|^2_2 + \| \rho \|^5_5 \\ =& -\int_\mathcal{O}  \int_0^z T(\cdot,\xi) \d \xi \cdot \divH v + \int_G Q(x)\beta(\rho) \cdot \rho - \int_G v|_{\Gamma_u} \cdot \nablaH \rho \cdot \rho.
    \end{aligned}
\end{equation*}
By H\"older's and Young's inequality we estimate
\begin{equation*}
    \bigl | \int_\mathcal{O}  \int_0^z T(\cdot,\xi) \d \xi \cdot \divH v \bigr | \leq \frac{1}{2} \| T \|^2_2 + \frac{1}{2} \| \nabla v \|^2_2 \leq C +  \frac{1}{2} \|\nabla v \|^2_2,
\end{equation*}
where we have used the embedding $\rL^\infty(\mathcal{O}) \hookrightarrow \rL^2(\mathcal{O})$ as well as \autoref{prop:linftylinfty} in the last step. Next, in view of $Q \in \rL^\infty(G)$ and
\autoref{prop:linftylinfty}, we obtain
\begin{equation*}
    \bigl | \int_G Q(x)\beta(\rho) \cdot \rho \bigr | \leq C \| \beta(\rho) \|_{\rL^1(G)} \leq C,
\end{equation*}
since $\beta$ is bounded. Finally, H\"older's inequality yields 
\begin{equation*}
    \bigl | \int_G v|_{\Gamma_u} \cdot \nablaH \rho \cdot \rho \bigr | \leq \| v|_{\Gamma_u} \|_2 \cdot \| \nablaH \rho \|_2 \cdot \| \rho \|_\infty \leq C \| v|_{\Gamma_u} \|_2 \cdot \| \nablaH \rho \|_2.
\end{equation*}
The velocity can be estimated by the trace theorem and an interpolation inequality, i.\ e.,
\begin{equation*}
    \| v|_{\Gamma_u} \|_{\rL^2(G)} \leq C \| v \|_{\rW^{\nicefrac{1}{2}}(\mathcal{O})} \leq C \| \nabla v \|^{\nicefrac{1}{2}}_{2} \cdot \| v \|^{\nicefrac{1}{2}}_2 
\end{equation*}
and we conclude using Young's inequality that 
\begin{equation*}
     \| v|_{\Gamma_u} \|_2 \cdot \| \nablaH \rho \|_2 \leq \frac{1}{4}\| \nabla v \|^2_2 + C \| v \|^2_2 + \frac{1}{2} \|\nablaH \rho \|_2^2.
\end{equation*}
Summarizing, there exists a constant $C>0$ such that 
\begin{equation*}
    \begin{aligned}
        \frac{1}{2}\dt \bigl ( \|v \|^2_2 + \| T \|^2_2 + \| \rho \|^2_2 \bigr) +\frac{1}{4} \| \nabla v \|^2_2 + \| \nabla T \|^2_2 + \frac{1}{4}\|\nablaH \rho \|^2_2 + \| T|_{\Gamma_b} \|^2_2 + \| \rho \|^5_5  \leq  C +C \| v \|_2^2.
    \end{aligned}
\end{equation*}
Integrating in time and applying Gronwall's inequality yields the claim.
\end{proof}

\begin{prop}[$\rL_t^\infty \rH^{1}_{\mathrm{x}}$-$\rL^2_t \rH^{2}_{\mathrm{x}}$-Estimates] \label{prop:linftyl2} \mbox{}\\
Let $(v,T,\rho)$ and $(v_0,T_0)$ be as in \eqref{eq:Assumptionsapriori} and \eqref{eq:Assumptionsapriori2}. Then there exists a continuous function $B_2$ on $[0,\tau]$, depending on $B_1$ and $t$, such that
\begin{equation*}
    \| \nabla v \|^2_{2} +  \| \nabla T \|^2_{2} + \| \nablaH \rho \|^2_{2} + \int_0^t \| \Delta v \|^2_{2} + \| \Delta T \|^2_2 + \| \DeltaH  \rho  \|^2_{2} \d s \leq B_2(t), \ \text{ for }\ t \in [0,\tau].
\end{equation*}
\end{prop}
\setcounter{stp}{0}
\begin{proof}
    The proof is divided into two steps.
    \begin{step}[$\rL_t^\infty \rH^{1}_{\mathrm{x}}$-$\rL^2_t \rH^{2}_{\mathrm{x}}$-bounds of the velocity] \mbox{}\\
      {\rm  By \autoref{lem:energyestimates}, we have
        \begin{equation*}
            \|\nablaH \int_0^z T(\cdot,\xi) \d \xi \|_{ \rL^2(0,\tau;\rL^2(\mathcal{O}))} \leq B_1.
        \end{equation*}
        Then, by the arguments in \cite[Section~6]{GGHHK:20b}, there exists a continuous function $B_3(t)$ depending on $\| v_0 \|_{\rH^1(\mathcal{O})}$, $B_1$, and $t$, such that
        \begin{equation*}
            \| \nabla v \|^2_2 + \int_0^t \| \Delta v \|^2_2 \d s \leq B_3(t), \ \text{ for } t \in [0,\tau].
        \end{equation*} }
    \end{step}
    \begin{step}[$\rL_t^\infty \rH^{1}_{\mathrm{x}}$-$\rL^2_t \rH^{2}_{\mathrm{x}}$-bounds of the temperature] \mbox{}\\
    { \rm Multiplying \eqref{eq: primitive + EBM simplified}$_3$ by $-\Delta T$, \eqref{eq: primitive + EBM simplified}$_5$ by $-\DeltaH \rho$, integrating over $\mathcal{O}$ and $G$ respectively, and adding the resulting equations yields
    \begin{equation*}
        \begin{aligned}
            &\dt \bigl ( \frac{1}{2}\| \nabla T \|^2_2   + \| \nablaH \rho \|^2_2+ \frac{1}{5} \| \rho \|^5_5 \bigr ) + \| \dt \rho \|^2_2 + \| \Delta T \|^2_2 + \| \DeltaH \rho \|^2_2 \\ &\le 
            \int_\mathcal{O} u \cdot \nabla T \cdot \Delta T + \int_G \bigl ( v|_{\Gamma_u} \cdot \nablaH \rho +Q(x) \beta(\rho) \bigr ) \cdot \bigl ( \dt \rho + \DeltaH \rho \bigr ) + \int_G (\dz T)|_{\Gamma_u}\cdot  \DeltaH \rho 
        \end{aligned}
    \end{equation*}
        The first term on the right-hand side can be estimated as the corresponding term in  \cite[Section~6]{Hieber2016}, i.\ e.,
        \begin{equation*}
            \bigl |  \int_\mathcal{O} u \cdot \nabla T \cdot \Delta T \bigr | \leq C \cdot \| v \|^2_{\rH^2(\mathcal{O})} \cdot \| \nabla T \|^2_2 + \frac{1}{2} \| \Delta T \|^2_2 .
        \end{equation*}
        Next, using H\"older's we obtain 
        \begin{equation*}
            \begin{aligned}
                \bigl |  \int_G v|_{\Gamma_u} \cdot \nablaH \rho \cdot \bigl ( \dt \rho + \DeltaH \rho \bigr ) \bigr | \leq \| v|_{\Gamma_u}\|_\infty \cdot \| \nablaH \rho\|_2 \cdot \bigl( \|  \dt \rho \|_2 + \| \DeltaH \rho\|_2 \bigr ).
            \end{aligned}
        \end{equation*}
        The embedding $\rW^{\nicefrac{3}{2}}(G) \hookrightarrow \rL^\infty(G)$ and trace theorem yield $   \| v|_{\Gamma_u}\|_\infty \leq C \| v \|_{\rH^2(\mathcal{O})}$, and Young's inequality implies
        \begin{equation*}
            \begin{aligned}
                 \| v|_{\Gamma_u}\|_\infty \cdot \| \nablaH \rho\|_2 \cdot \bigl( \|  \dt \rho \|_2 + \| \DeltaH \rho\|_2 \bigr ) \leq C \| v \|^2_{\rH^2(\mathcal{O})}  \cdot \| \nablaH \rho\|^2_2 + \frac{1}{4} \bigl (  \|  \dt \rho \|^2_2 + \| \DeltaH \rho\|^2_2 \bigr ).
            \end{aligned}
        \end{equation*}
        Similarly, we estimate
        \begin{equation*}
            \begin{aligned}
                \bigr | \int_G Q(x) \beta(\rho) \cdot \bigr (\dt \rho + \DeltaH \rho \bigl ) \bigl | &\leq C \| Q(x) \beta(\rho) \|^2_2 + \frac{1}{4}  \bigl (  \|  \dt \rho \|^2_2 + \| \DeltaH \rho\|^2_2 \bigr ) \\ &\leq C + \frac{1}{4}  \bigl (  \|  \dt \rho \|^2_2 + \| \DeltaH \rho\|^2_2 \bigr ),
            \end{aligned}
        \end{equation*}
        where $\| Q(x) \beta(\rho) \|^2_2 \leq C$ by the boundedness of $Q$ and $\beta$ respectively. Concerning the last term, H\"older's inequality and trace theorem yield
        \begin{equation*}
            \bigl | \int_G (\dz T)|_{\Gamma_u}\cdot \DeltaH \rho  \bigr | \leq \| \nabla T \|_{\rW^{\nicefrac{1}{2}}(\mathcal{O})} \cdot \|  \DeltaH \rho \|_2.
        \end{equation*}
        Here, we estimate the full temperature $T$ using an interpolation inequality to obtain
        \begin{equation*}
            \begin{aligned}
                \| \nabla T \|_{\rW^{\nicefrac{1}{2}}(\mathcal{O})} \leq C \| \nabla T \|^{\nicefrac{1}{2}}_{\rH^1(\mathcal{O})} \cdot \| \nabla T \|^{\nicefrac{1}{2}}_2 \leq C \bigl ( \| \nabla T \|_2 + \| \nabla T \|^{\nicefrac{1}{2}}_2 \cdot \| \Delta T \|^{\nicefrac{1}{2}}_2 \bigl ).
            \end{aligned}
        \end{equation*}
        Young's inequality implies
        \begin{equation*}
            \| \nabla T \|_{\rW^{\nicefrac{1}{2}}(\mathcal{O})} \cdot \|  \DeltaH \rho \|_2 \leq C(\eps,\eps') \| \nabla T \|^2_2 + (\eps + \eps') \|  \DeltaH \rho \|_2^2  + \eps' \| \Delta T \|^2_2,
        \end{equation*}
        for any $\eps$, $\eps'>0$. Note that the constant $C(\eps,\eps')$ blows up when $\eps$, $\eps'\to 0$. Choosing $\eps$, $\eps'$ appropriately small, we conclude that there exists a constant $C>0$ such that 
         \begin{equation*}
        \begin{aligned}
            &\dt \bigl ( \frac{1}{2}\| \nabla T \|^2_2 +\frac{1}{2} \| T|_{\Gamma_b}\|^2_2   + \| \nablaH \rho \|^2_2+ \frac{1}{5} \| \rho \|^5_5 \bigr ) +\frac{3}{4} \| \dt \rho \|^2_2 +\frac{1}{2}^- \| \Delta T \|^2_2 +\frac{1}{2}^- \| \DeltaH \rho \|^2_2 \\ &\leq C \bigl (  \| v \|^2_{\rH^2(\mathcal{O})}  \cdot  \| \nabla T \|^2_2  + \| v \|^2_{\rH^2(\mathcal{O})}  \cdot \| \nablaH \rho\|^2_2  +1 \bigr ).      
        \end{aligned}
    \end{equation*}
    Integrating in time and applying Gronwall's inequality yield the claim. 
        }
    \end{step}
\end{proof}
We are now in position to prove the global existence of the solutions.

\begin{proof}[Proof of \autoref{thm: local wp}(b)]
By an iterative application of \autoref{thm: local wp}(a), local solutions $(v,T,\rho)$ of \eqref{eq: primitive + EBM simplified} exists on a maximal time interval $J_{\mathrm{max}}=[0,a_{\mathrm{max}})$. If $a_{\mathrm{max}} < \tau$, i.\ e., the solution does not exist globally, the maximal existence time $a_{\mathrm{max}}$ is characterized by the condition
\begin{equation}\label{rem:max time ex}
       \lim\limits_{\tau \to a_\mathrm{max}} \| ( v,T,\rho ) \|_{\E_{1,\tau}}  = \infty.
\end{equation}
Therefore, to prove global well-posedness for $p=2$, we show that the bounds derived in \autoref{prop:linftyl2} are sufficient to control the maximal regularity norm of the solution $(v,T,\rho)$. By maximal $\rLp$-regularity, there exists a constant $C>0$ such that, for $\tau < a_{\mathrm{max}}$, it holds 
\begin{equation*}
  \begin{aligned}
     \| (v,T,\rho) \|_{\E_{1,\tau}} &\leq  
    C \bigl ( \| \nablaH \int_0^z T(\cdot, \xi) \d \xi - v \cdot \nablaH v - w(v) \cdot \dz v \|_{\E_{0,\tau}^v} \\& \quad + \| (-v \cdot \nablaH T - w(v) \cdot \dz T,  R(x,\rho)- v|_{\Gamma_u} \cdot \nablaH \rho) \|_{\E_{0,\tau}^T} \bigr ).
  \end{aligned}
\end{equation*}
Concerning the convection term in the primitive equations as well as the transport term in the temperature equation, we verify that 
\begin{equation*}
    \begin{aligned}
        \| v \cdot \nablaH v + w(v) \cdot \dz v \|^2_{\E_{0,\tau}^v} &\leq C \left \|v \right \|^2_{\rL^\infty(0,\tau;\rH^{1}(\mathcal{O}))} \cdot  \left \|v \right \|^2_{\rL^2(0,\tau;\rH^{2}(\mathcal{O}))} \ \text{ and }\\ \| u \cdot \nabla T \|^2_{\rLp((0,\tau)\times \mathcal{O})} &\leq C \bigl( \left \|v \right \|^2_{\rL^\infty(0,\tau;\rH^{1}(\mathcal{O}))} +  \left \|T \right \|^2_{\rL^2(0,\tau;\rH^{2}(\mathcal{O}))}\bigr ) \cdot \bigl (  \left \|T \right \|^2_{\rL^\infty(0,\tau;\rH^{1}(\mathcal{O}))} +  \left \|v \right \|^2_{\rL^2(0,\tau;\rH^{2}(\mathcal{O}))} \bigr ).
    \end{aligned}
\end{equation*}
Moreover, $\| \nablaH \int_0^z T(\cdot, \xi) \d \xi \|_{\E_{0,\tau}^v} \leq C\| T \|_{\rL^2(0,\tau;\rH^1(\mathcal{O}))}$ and it remains to bound $R(x,\rho)$ as well as $v|_{\Gamma_u} \cdot \nablaH \rho$. Arguing as in the proof of \autoref{lem: esimates nonlinearities version 2} and using the embedding $\rH^1(G) \hookrightarrow \rL^8(G)$ yields
\begin{equation*}
       \begin{aligned}
           \| \rho^4 \|_{\rL^2((0,\tau) \times G)} \leq C \| \rho\|^4_{\rL^{8}(0,\tau;\rL^{8}(G))} \le C \| \rho\|^4_{\rL^{\infty}(0,\tau;\rH^{1}(G))}. 
       \end{aligned}
\end{equation*}
Moreover, by the boundedness of $Q$ and $\beta$ it is immediate that
    \begin{equation*}
          \| Q(x)\beta(\rho) \|_{\rL^2((0,\tau) \times G)} \leq C.
    \end{equation*}
    Finally, by similar methods as in \autoref{lem: esimates nonlinearities version 2} we obtain for the transport term on the surface 
    \begin{equation*}
        \| v|_{\Gamma_u} \cdot \nablaH  \rho \|_{\rL^2((0,\tau) \times G)} \leq   C \| v \|_{\rL^\infty(0,\tau;\rH^{\nicefrac{3}{4}}(G))} \cdot \|   \nablaH\rho\|_{\rL^2(0,\tau;\rW^{\nicefrac{1}{2}}(G))} \leq C \| v \|_{\rL^\infty(0,\tau;\rH^1(G))} \cdot \|   \rho\|_{\rL^2(0,\tau;\rH^{2}(G))}.
    \end{equation*}
Using \autoref{prop:linftyl2} we conclude that there exists a constant $C>0$ such that
    \begin{equation*}
         \| (v,T,\rho) \|_{\E_{1,\tau}} \leq C,
    \end{equation*}
    and hence global existence within the $\rL^2$-framework follows.

    In order to proof global existence for general $p\geq 2$ consider the local solution $(v,T,\rho)\in \E_{1,\tau}$ for initial data $(v_0,T_0)$ satisfying assumption (A) as well as $(v_0,T_0,T_0|_{\Gamma_u}) \in \rL^\infty(\mathcal{O}) \times \rL^\infty(\mathcal{O}) \times \rL^\infty(G)$. Here, $\tau$ denotes the  maximal time of existence characterized by \eqref{rem:max time ex}. The smoothing property of
    parabolic equations implies that $(v(\delta),T(\delta))$ satisfies \eqref{eq:Assumptionsapriori2}  all $0<\delta <\tau$. Hence, the solution $(v,T,\rho)$ exists globally within the $\rL^2$-framework. Performing a bootstrap argument as in \cite[Section~6.2]{Hieber2016} yields that the solution also exists globally within the $\rLp$-framework.
\end{proof}

\subsection{Higher regularity} \mbox{}\\
The remainder of this section is dedicated to the proof of \autoref{cor:regularity}. Note that the assertions in \autoref{cor:regularity} are based on Angenent's parameter trick, see \cite{A:90}.

\begin{proof}[Proof of \autoref{cor:regularity}]
To simplify our notation, we only consider $G = \R^2$ and identify $\mathcal{O}$ with $G \times (0,1)$. We start by showing interior regularity of the global, unique solutions $v$ and $T$ of \eqref{eq: primitive + EBM simplified} subject to \eqref{eq:bc}. For this purpose, we fix $(\tau_0,x_0) \in (0,\tau) \times \mathcal{O}$ and recall that regularity is a local property, so we need only to show regularity of $(v,T)$ in $(\tau_0-r, \tau_0+r) \times \rB_r(x_0)$, for $r$ small enough. Choose $R>0$ such that $\rB_{3R}(x_0) \in \mathcal{O}$, and let $\zeta$ be a smooth cut-off function satisfying $\zeta(x) =1$ for $|x-x_0|<R$, $\zeta(x) =0$ for $|x-x_0|>2R$, and $0 \leq \zeta(x) \leq 1$ elsewhere. Define the truncated shift by
   \begin{align*}
       \theta_{\lambda,\xi} (t,x) := (t+ \lambda t, x+ t\xi \zeta(x)) \ \text{ for }\  (t,x) \in \mathcal{O} \times (0,\tau).
   \end{align*}
Then $\theta_{\lambda,\xi} \colon (0,\tau) \times \mathcal{O}$ is a diffeomorphism of class $\rC^\infty$, so that the map 
\begin{equation*}
\theta \colon (-r,r) \times \rB_r(0) \to \mathrm{Diff}^\infty((0,\tau) \times \mathcal{O})
\end{equation*}
is well-defined for $r$ sufficiently small. Next, we see that the push-forward operator $\Theta$ defined by 
    \begin{align*}
        \Theta_{\lambda,\xi} \binom{v}{T}
        := \binom{v \circ \theta_{\lambda,\xi}}{T \circ \theta_{\lambda,\xi}} \ \text{ for }\ (\lambda,\xi) \in  (-r,r) \times \rB_r(0)
    \end{align*}
is an isomorphism on $\E^v_{1,\tau} \times \E^{T}_{1,\tau}$, provided $r$ is small enough. Finally, define $\tilde{\rH}\colon \E_{1,\tau} \to \E_{0,\tau} \times \F_\tau \times \rX_\gamma$ such that $\tilde{\rH}(v,T,\rho)= (0,0,0,0,0,0,v_0,T_0)$ is equivalent to \eqref{eq: primitive + EBM simplified} and 
\begin{equation*}
    \rH(\lambda,\xi, v,T,\rho) := \Theta_{\lambda,\xi}\tilde{\rH}(\Theta^{-1}_{\lambda,\xi}(v,T)^\top, \rho).
\end{equation*}
Here, $\F_\tau$ denotes the optimal space concerning the boundary data of $v$ and $T$ given by $\F_\tau = (\F^v_\tau)^2 \times \F^T_\tau$ where
\begin{equation*}
    \F^v_\tau = \rW^{\nicefrac{1}{2}-\nicefrac{1}{2p},p}(0,\tau;\rLp(G;\R^2)) \cap \rLp(0,\tau;\rW_\per^{1-\nicefrac{1}{p},p}(G;\R^2)) 
\end{equation*}
and $\F^T_\tau$ is defined similarly (for scalar valued functions). Note that $\rH\in \rC^1$ and therefore consider the Fr\'echet derivative of $\rH$ with respect to $(v,T)$ evaluated at $(0,0,\hat{v},\hat{T},\rho)$, i.\ e.,
\begin{equation*}
    \D_{(v,T)}\rH(0,0,\hat{v},\hat{T},\rho) \colon \E_{1,\tau} \to \E_{0,\tau} \times \F_\tau \times \rX_\gamma.
\end{equation*}
Then $\rH(0,0,\hat{v},\hat{T},\rho) =  (0,0,0,0,0,0,v_0,T_0)$ and $  \D_{(v,T)}\rH(0,0,\hat{v},\hat{T},\rho)$ is again an isomorphism by the considerations in \autoref{sec: lin and loc} and since the maps corresponding to the non-linear terms are sufficiently smooth. Indeed, by \cite[Lemma~6.4]{GGHHK:20b} \etz the maps
\begin{equation*}
    \begin{aligned}
        \mathcal{F}^v \colon \E^v_{1,\tau} \to \E^v_{0,\tau}, \ v \mapsto v \cdot \nablaH v + w(v) \cdot \dz v   \text{ and } \mathcal{F}^T \colon \E^v_{1,\tau} \times \E^T_{1,\tau} \to \E^T_{0,\tau}, \ (v,T) \mapsto v \cdot \nablaH T + w(v) \cdot \dz T 
     \end{aligned}
\end{equation*}
are continuously differentiable and even real analytic. Hence, by the implicit function theorem, there exists $\delta >0$ and a function  $\Phi\in \rC^\infty((-\delta, \delta) \times \rB_\delta(0); \E^v_{1,\tau} \times \E^T_{1,\tau})$ such that 
\begin{equation*}
    \rH(\lambda,\xi,\Phi(\lambda,\xi),\rho)=0,
\end{equation*}
for all $\lambda \in (-\delta, \delta)$ and $\xi \in \rB_\delta(0)$. Now it holds true that
\begin{equation*}
    (\lambda,\xi) \mapsto \binom{v}{T}(\tau_0+\lambda,x_0 +\tau_0\xi)
\end{equation*}
is of class $\rC^\infty$ near $(\tau_0,x_0)$. Thus, we have $v, T \in \rC^\infty((0,\tau ) \times \mathcal{O})$. Noticing that $\Phi$ is also of class $\rC^\omega$, we deduce in a similar way that $v$ and $T$ are real analytic.

In order to show that $\rho \colon G \to \R$ is smooth and real analytic we fix a point in the upper boundary, i.\ e., $(\tau_0,x_{0},1) \in (0,\tau) \times G$ and show regularity of $\rho$ in a neighbourhood of $(\tau_0,x_{0},1)$. Note that the map 
\begin{equation*}
      \mathcal{F}^\rho \colon \E^v_{1,\tau} \times \E^\rho_{1,\tau} \to \E^\rho_{0,\tau}, \ (v,\rho) \mapsto v|_{\Gamma_u} \cdot \nablaH \rho - R(x,\rho)
\end{equation*}
is continuously differentiable and real analytic and therefore arguing as in the proof for interior regularity yields the claim.
\end{proof}

\section{Global well-posedness: stochastic model}\label{sec: stoch} 

We start by stating the stochastic maximal regularity result concerning the linear system \eqref{eq: stoch eq}, which goes back to \cite{Veraar2012a,Veraar2012b}. Note that \eqref{eq: stoch eq} can be considered without shift for finite time $\tau$.

\begin{prop}[Stochastic maximal regularity]\label{prop: stoch max reg}
    Let $0< \tau < \infty$. Then for any strongly $\mathcal{F}_0$-measurable $H =(0 , H_\rho) \in \rL^2(\Omega\times (0,\tau);\rL^2(\mathcal{U}; \D(-A+\omega)^{\nicefrac{1}{2}}))$ the stochastic convolution given by
    \begin{equation}
        \label{eq: stoch convo}
        Z(t) := \int_0^t \mathrm{e}^{-(t-s)A} H(s) \d W(s) 
    \end{equation}
    is well-defined, $\mathcal{F}$-adapted and defines the unique solution to \eqref{eq: stoch eq} satisfying
    \begin{equation*}
        Z \in \rL^2\bigl (\Omega;\rH^{\theta}(0,\tau;\D(-A+\omega)^{1-\theta}) \bigr ) \cap \rL^2\bigl (\Omega;\rC([0,\tau];\rX_\gamma^{(T,\rho)})\bigr ) \ \text{ for any } \theta \in [0,1/2).
    \end{equation*}
\end{prop}
\noindent
We are now in position to proof global well-posedness of the system \eqref{eq: primitive + EBM simplified stohcastic}.
\begin{proof}[Proof of \autoref{thm: global stoch}]
    Local, maximal existence can be deduced similarly to \autoref{sec: lin and loc} by showing that the non-linear terms on the surface can be controlled in terms of the solution. Concerning the transport term, we observe in view of H\"older's and Jensen's inequalities as well as the embedding \(\E_{1,\tau}^v \hookrightarrow \rL^4(0,\tau; \rL^4(\mathcal{O}))\) that
    \[
    \begin{aligned}
        \| \bar{v} \cdot \nablaH Z_\rho \|_{\rL^2((0,\tau) \times G)} &\leq C \| \bar{v} \|_{\rL^4(0,\tau; \rL^4(G))} \cdot \| \nablaH Z_\rho \|_{\rL^4(0,\tau;\rL^4(G))} \\
        &\leq C \| v \|_{\rL^4(0,\tau; \rL^4(\mathcal{O}))} \cdot \| Z_\rho \|_{\rL^4(0,\tau;\rH^{1,4}(G))} \\
        &\leq C \| v \|_{\E_{1,\tau}^v} \cdot \| Z_\rho \|_{\rL^4(0,\tau;\rH^{1,4}(G))}.
    \end{aligned}
    \]
    By interpolation and H\"older's inequality, we obtain
    \[
    \|Z_\rho\|_{\rL^4(0,\tau;\rH^{1,4}(G))} \leq C \| Z_\rho \|_{\rL^\infty(0,\tau;\rH^{1}(G))}^{\nicefrac{1}{2}} \cdot \| Z_\rho \|_{\rL^2(0,\tau;\rH^{2}(G))}^{\nicefrac{1}{2}},
    \] 
    and therefore Young's inequality yields
    \begin{equation*}
         \| \bar{v} \cdot \nablaH Z_\rho \|_{\rL^2((0,\tau) \times G)} \leq C \bigl (\| v \|_{\E^v_{1,\tau}} \cdot \| Z_\rho \|_{\rL^\infty(0,\tau;\rH^1(G))} + \| v \|_{\E^v_{1,\tau}} \cdot \| Z_\rho \|_{\rL^2(0,\tau;\rH^2(G))} \bigr ).
    \end{equation*}
    Next, it follows from \eqref{eq: radiation}  that
    \[
    \| R(x,\tilde{\rho}+ Z_\rho) \|_{\rL^2((0,\tau) \times G)} \leq \| Q(x)\beta(\tilde{\rho}+ Z_\rho ) \|_{\rL^2((0,\tau) \times G)} + \| (\tilde{\rho}+Z_\rho )^4 \|_{\rL^2((0,\tau) \times G)}.
    \]
    The first term on the right-hand side can be estimated as in the proof of \autoref{lem: esimates nonlinearities version 2}. Next, we have
    \[
    \| (\tilde{\rho}+Z_\rho )^4 \|_{\rL^2((0,\tau) \times G)} \leq C \Bigl( \|\tilde{\rho}^4\|_{\rL^2((0,\tau) \times G)} + \|Z_\rho^4 \|_{\rL^2((0,\tau) \times G)} \Bigr),
    \]
    and, similarly, the first term on the right-hand side can be estimated as in \autoref{lem: esimates nonlinearities version 2}. For the remaining term, H\"older's inequality implies that
    \[
    \|Z_\rho^4 \|_{\rL^2((0,\tau) \times G)} \leq C \|Z_\rho\|^4_{\rL^8(0,\tau; \rL^8(G))} \leq C \|Z_\rho\|_{\rL^\infty(0,\tau; \rH^1(G))}^{4}.
    \]
    Arguing as in the proof of \autoref{thm: local wp}(a) concludes the existence of a  unique, local, maximal solution of system \eqref{eq: primitive + EBM simplified1}. To show that  the local solution extends to a global one, we adapt the a priori estimates derived in \autoref{sec: global} accordingly. Multiplying \eqref{eq: primitive + EBM simplified1}$_1$ by $v$ as well as \eqref{eq: primitive + EBM simplified1}$_3$ by $\tilde{T}$ and integrating over the respective domains yields
    \begin{equation*}
        \begin{aligned}
            &\frac{1}{2}\dt \bigl ( \| v \|^2_2 + \| \tilde{T} \|^2_2 + \| \tilde{\rho} \|^2_2 \bigr ) + \| \nabla v \|^2_2 + \| \nabla \tilde{T} \|^2_2 + \| \nablaH \tilde{\rho} \|^2_2 + \| \tilde{\rho} \|^5_5  \\ &\le   \int_\mathcal{O} \nablaH \int_0^z \tilde{T}(\cdot,\xi) \d \xi \cdot v  - \int_G \vbar \cdot \nablaH Z_\rho\cdot \tilde{\rho}  +\int_G Q(x)\beta(\tilde{\rho} + Z_\rho)  \cdot \tilde{\rho} -\int_G|\tilde{\rho}+Z_\rho|^3 \cdot Z_\rho \cdot \tilde{\rho}.
        \end{aligned}
    \end{equation*}
    Using H\"older's as well as Young's inequality yields 
    \begin{equation*}
        \bigl | \int_\mathcal{O} \nablaH \int_0^z \tilde{T}(\cdot,\xi) \d \xi \cdot v \bigr | \leq \eps \| \nabla T \|^2_2 + C \| v \|^2_2,
    \end{equation*}
    for an appropriate constant $C>0$. Concerning the second addend, H\"older's, Young's and interpolation inequality yield the estimate
    \begin{equation*}
        \bigl | \int_G \vbar \cdot \nablaH Z_\rho \cdot \tilde{\rho} \bigr | \leq  \eps \bigl ( \| v \|^2_{\rH^1} +  \| \tilde{\rho} \|^2_{\rH^1}\bigr ) + C \bigl \| Z_\rho \|^2_{\rH^1} ( \| v \|^2_2 + \| \tilde{\rho} \|^2_2 \bigr )
    \end{equation*}
   Next, note that the third addend can be estimated by similar arguments as in \autoref{sec: global}, using boundedness of $Q$ and $\beta$. Finally, H\"older's and Young's inequality imply
   \begin{equation*}
       \bigl | \int_G|\tilde{\rho}+Z_\rho|^3 \cdot Z_\rho \cdot \tilde{\rho} \bigr | \leq \| \tilde{\rho}+ Z_\rho \|^3_5 \cdot \|Z_\rho \|_5 \cdot \| \tilde{\rho}\|_5 \leq \| \tilde{\rho} \|^4_5 \cdot \| Z_\rho \|_5 + \| Z_\rho \|^4_5 \cdot \| \tilde{\rho} \|_5 \leq \eps \| \tilde{\rho} \|^5_5 + C \|Z_\rho \|^5_5.
   \end{equation*}
   Integrating in time, Gronwall's inequality and \autoref{prop: stoch max reg} then yield the energy bound
   \begin{equation*}
          \|{v}\|_{2}^2 +\|\tilde{T}\|_{2}^2 + \|\tilde{\rho}\|_{2}^2+ +  \| \tilde{\rho} \|_{5}^5  2 \int_0^t \|\nabla v\|_{2}^2 +\|\nabla \tilde{T}\|_{2}^2 + \| \nablaH \tilde{\rho}\|_{2}^2 \leq B_1(t), \ \text{ for all }t \in [0,\tau]
   \end{equation*}
   where $B_1$ is a continuous function depending on $\|v_0\|_{\rH^1}$, $\|T_0\|_{\rH^1}$, $\|T_0|_{\Gamma_u} \|_{\rH^1}$ and $t$. From the energy bounds we conclude that
   \begin{equation*}
         \|\nablaH \int_0^z T(\cdot,\xi) \d \xi \|_{ \rL^2((0,t) \times \mathcal{O})} \leq B_1(t)
   \end{equation*}
   and therefore, arguing as in \autoref{sec: global} leads to $v \in \rL^2(0,\tau;\rH_\per^2(\mathcal{O};\R^2)) \cap \rL^\infty(0,\tau;\rH_\per^1(\mathcal{O};\R^2))$. To obtain similar bounds for the temperature, we multiply \eqref{eq: primitive + EBM simplified1}$_3$ by $-\Delta T$, \eqref{eq: primitive + EBM simplified1}$_5$ by $-\DeltaH \rho$, integrate over $\mathcal{O}$ and $G$ respectively, and add the resulting equations to obtain
      \begin{equation*}
        \begin{aligned}
            &\dt \bigl ( \frac{1}{2}\| \nabla \tilde{T} \|^2_2   + \| \nablaH \tilde{\rho} \|^2_2+ \frac{1}{5} \| \tilde{\rho} \|^5_5 \bigr ) + \| \dt \tilde{\rho} \|^2_2 + \| \Delta \tilde{T} \|^2_2 + \| \DeltaH \tilde{\rho} \|^2_2 \\ &\le 
            \int_\mathcal{O} u \cdot \nabla \tilde{T} \cdot \Delta \tilde{T} + \int_G \bigl ( \vbar \cdot \nablaH (\tilde{\rho} +Z_\rho) +Q(x) \beta(\tilde{\rho}+Z_\rho) \bigr ) \cdot \bigl ( \dt \tilde{\rho} + \DeltaH \tilde{\rho} \bigr ) + \int_G (\dz \tilde{T})|_{\Gamma_u}\cdot  \DeltaH \tilde{\rho} \\
            & \quad - \int_G |3\tilde{\rho}^2 \cdot  Z_\rho + 3 \tilde{\rho} \cdot Z_\rho + Z_\rho^3| \cdot  (\tilde{\rho}+Z_\rho) \cdot ( \dt \tilde{\rho} + \DeltaH \tilde{\rho}).
        \end{aligned}
    \end{equation*}
    Since $Z_\rho \in \rL^2(0,\tau;\rH_\per^2(\mathcal{O})) \cap \rL^\infty(0,\tau;\rH_\per^1(\mathcal{O}))$, the terms in the second line of the above inequality can be estimated as in \autoref{sec: global}. Using H\"older's and Young's inequality, we obtain for the term in the third line
    \begin{equation*}
        \bigl |\int_G |3\tilde{\rho}^2 \cdot  Z_\rho + 3 \tilde{\rho} \cdot Z_\rho + Z_\rho^3| \cdot  (\tilde{\rho}+Z_\rho) \cdot ( \dt \tilde{\rho} + \DeltaH \tilde{\rho}) \bigr | \leq C \|  |3\tilde{\rho}^2 \cdot  Z_\rho + 3 \tilde{\rho} \cdot Z_\rho + Z_\rho^3| \cdot  (\tilde{\rho}+Z_\rho) \|^2_2 + \eps \bigl (\| \dt \tilde{\rho} \|^2_2 + \| \DeltaH \tilde{\rho} \|^2_2 \bigr ).
    \end{equation*}
    Here we further estimate using interpolation to obtain
    \begin{equation*}
        \begin{aligned}
            \| \tilde{\rho}^3 \cdot Z_\rho \|^2_2 &\leq C \| \tilde{\rho} \|^6_8 \cdot \| Z_\rho \|^2_8 \leq C \| \tilde{\rho} \|^{\nicefrac{15}{4}}_5 \cdot \| \tilde{\rho} \|^{\nicefrac{9}{4}}_{\rH^1} \cdot \| Z_\rho \|^2_{\rH^1} \leq C \| \tilde{\rho} \|^2_{\rH^1} \cdot \| Z_\rho \|^{\nicefrac{8}{3}}_{\rH^1} \cdot \| \tilde{\rho} \|^5_5 + C \| \tilde{\rho}\|_{\rH^1} \cdot \| \tilde{\rho} \|^2_{\rH^1}, \\
             \| \tilde{\rho}^2 \cdot Z_\rho^2 \|^2_2 &\leq \| \tilde{\rho} \|^4_4 \cdot  \| Z_\rho \|^4_4 \leq C \| \tilde{\rho} \|^5_5 + C \| Z_\rho \|^{20}_{\rH^1}
        \end{aligned}
    \end{equation*}
    and by similar arguments it holds true that
    \begin{equation*}
        \| \tilde{\rho} \cdot Z^3_\rho \|^2_2 + \| \tilde{\rho} \cdot Z_\rho^2 \|^2_2 + \| Z_\rho^4 \|^2_2\leq C \| \tilde{\rho} \|^5_5 + \| Z_\rho \|^\delta_{\rH^1} + C, 
    \end{equation*}
    for some $\delta \in \mathbb{N}$. By  $Z_\rho \in \rL^\infty(0,\tau;\rH_\per^1(\mathcal{O}))$, integrating in time and Gronwall's inequality, we conclude the existence of a continuous function $B_2$ such that
      \begin{equation*}
    \| \nabla \tilde{T} \|^2_{2} + \| \nablaH \tilde{\rho} \|^2_{2} + \| \tilde{\rho} \|^5_5 \int_0^t  \| \Delta \tilde{T} \|^2_2 + \| \DeltaH  \tilde{\rho}  \|^2_{2} \leq B_3(t)
\end{equation*}
for all $0\leq t \leq \tau$. Finally by maximal $\rL^2$-regularity the solution $(v,\tilde{T},\tilde{\rho})$ admits the estimate
\begin{equation*}
    \begin{aligned}
        \| (v,\tilde{T},\tilde{\rho}) \|_{\E_{1,\tau}} &\leq  
    C \bigl ( \| \nablaH \int_0^z \tilde{T}(\cdot, \xi) \d \xi - v \cdot \nablaH v - w(v) \cdot \dz v \|_{\E_{0,\tau}^v} \\& \quad + \| (-v \cdot \nablaH \tilde{T} - w(v) \cdot \dz \tilde{T},  R(x,\tilde{\rho}+Z_\rho)-\vbar \cdot \nablaH (\tilde{\rho}+ Z_\rho) \|_{\E_{0,\tau}^T} \bigr ).
    \end{aligned}
\end{equation*}
Arguing as in the proof of \autoref{thm: local wp}(b), we see that the bounds for $(v,\tilde{T},\tilde{\rho})$ derived above together with \autoref{prop: stoch max reg} are sufficient to control the solution norm of $(v,\tilde{T},\tilde{\rho})$.
\end{proof}

 
{\bf Data availability}
\small{We do not analyse or generate any datasets, because our work proceeds within a theoretical
and mathematical approach.}

{\bf Conflict of interest}
\small{The authors declare that there is no conflict of interest.}


\begin{thebibliography}{10}

\bibitem{Agresti2022a}
A.~Agresti and M.~Veraar.
\newblock Nonlinear parabolic stochastic evolution equations in critical spaces
  part {I}. {S}tochastic maximal regularity and local existence.
\newblock {\em Nonlinearity}, 35(8):4100--4210, 2022.

\bibitem{Agresti2022b}
A.~Agresti and M.~Veraar.
\newblock Nonlinear parabolic stochastic evolution equations in critical spaces
  part {II}: {B}low-up criteria and instataneous regularization.
\newblock {\em J. Evol. Equ.}, 22(2):Paper No. 56, 96, 2022.

\bibitem{Ama:19}
H.~Amann.
\newblock {\em Linear and Quasilinear Parabolic Problems. Vol.~II. Function
  Spaces}, volume 109.
\newblock Birkh\"auser/Springer, Cham, 2019.

\bibitem{A:90}
S.~Angenent.
\newblock Parabolic equations for curves on surfaces part i. curves with
  p-integrable curvature.
\newblock {\em Ann. of Math}, 132(3):451--483, 1990.

\bibitem{Ashwin2020}
P.~Ashwin and A.~von~der Heydt.
\newblock Extreme sensitivity and climate tipping points.
\newblock {\em J. Stat. Phys.}, 179(5-6):1531--1552, 2020.

\bibitem{Binz2024}
Tim Binz, Matthias Hieber, Amru Hussein, and Martin Saal.
\newblock The primitive equations with stochastic wind driven boundary
  conditions.
\newblock {\em J. Math. Pures Appl. (9)}, 183:76--101, 2024.

\bibitem{Bonaccorsi2006}
S.~Bonaccorsi and G.~Ziglio.
\newblock A semigroup approach to stochastic dynamical boundary value problems.
\newblock In {\em Systems, control, modeling and optimization}, volume 202 of
  {\em IFIP Int. Fed. Inf. Process.}, pages 55--65. Springer, New York, 2006.

\bibitem{Russo2015}
Z.~Brze\'zniak, B.~Goldys, S.~Peszat, and F.~Russo.
\newblock Second order {PDE}s with {D}irichlet white noise boundary conditions.
\newblock {\em J. Evol. Equ.}, 15(1):1--26, 2015.

\bibitem{Budyko1969}
M.~Budyko.
\newblock The effect of solar radiation variations on the climate of the earth.
\newblock {\em Tellus}, 21(5):611--619, 1969.

\bibitem{Cannarsa23}
P.~Cannarsa, V.~Lucarini, P.~Martinez, C.~Urbani, and J.~Vancostenoble.
\newblock Analysis of a two-layer energy balance model: long time behavior and
  greenhouse effect.
\newblock {\em Chaos}, 33(11):Paper No. 113111, 34, 2023.

\bibitem{CT:07}
C.~Cao and E.~Titi.
\newblock Global well-posedness of the three-dimensional viscous primitive
  equations of large scale ocean and atmosphere dynamics.
\newblock {\em Ann. of Math. (2)}, 166(1):245--267, 2007.

\bibitem{Cerrai2011}
S.~Cerrai and M.~Freidlin.
\newblock Fast transport asymptotics for stochastic {RDE}s with boundary noise.
\newblock {\em Ann. Probab.}, 39(1):369--405, 2011.

\bibitem{DPD02}
G.~Da~Prato and A.~Debussche.
\newblock Two-dimensional {N}avier-{S}tokes equations driven by a space-time
  white noise.
\newblock {\em J. Funct. Anal.}, 196(1):180--210, 2002.

\bibitem{DPD03}
G.~Da~Prato and A.~Debussche.
\newblock Strong solutions to the stochastic quantization equations.
\newblock {\em Ann. Probab.}, 31(4):1900--1916, 2003.

\bibitem{Deubussche2012}
A.~Debussche, N.~Glatt-Holtz, R.~Temam, and M.~Ziane.
\newblock Global existence and regularity for the 3{D} stochastic primitive
  equations of the ocean and atmosphere with multiplicative white noise.
\newblock {\em Nonlinearity}, 25(7):2093--2118, 2012.

\bibitem{Debussche2007}
Arnaud Debussche, Marco Fuhrman, and Gianmario Tessitore.
\newblock Optimal control of a stochastic heat equation with boundary-noise and
  boundary-control.
\newblock {\em ESAIM Control Optim. Calc. Var.}, 13(1):178--205, 2007.

\bibitem{DelSarto24a}
G.~Del~Sarto, J.~Br\"ocker, F.~Flandoli, and T.~Kuna.
\newblock Variational techniques for a one-dimensional energy balance model.
\newblock {\em Nonlinear Processes in Geophysics}, 31(1):137--150, 2024.

\bibitem{DelSarto24}
G.~Del~Sarto and F.~Flandoli.
\newblock A non-autonomous framework for climate change and extreme weather
  events increase in a stochastic energy balance model.
\newblock {\em Chaos}, 34(9):Paper No. 093122, 16, 2024.

\bibitem{MR2047641}
R.~Denk, G.~Dore, M.~Hieber, J.~Pr\"uss, and A.~Venni.
\newblock New thoughts on old results of {R}. {T}.\ {S}eeley.
\newblock {\em Math. Ann.}, 328(4):545--583, 2004.

\bibitem{DHP03}
R.~Denk, M.~Hieber, and J.~Pr\"uss.
\newblock {$\mathcal{R}$}-boundedness, {F}ourier multipliers and problems of
  elliptic and parabolic type.
\newblock {\em Mem. Amer. Math. Soc.}, 166(788):viii+114, 2003.

\bibitem{Denk2008}
R.~Denk, J.~Pr\"uss, and R.~Zacher.
\newblock Maximal {$\rL_p$}-regularity of parabolic problems with boundary
  dynamics of relaxation type.
\newblock {\em J. Funct. Anal.}, 255(11):3149--3187, 2008.

\bibitem{Diaz2022}
G.~D\'iaz and J.I.~D\'iaz.
\newblock Stochastic energy balance climate models with {L}egendre weighted
  diffusion and an additive cylindrical {W}iener process forcing.
\newblock {\em Discrete Contin. Dyn. Syst. Ser. S}, 15(10):2837--2870, 2022.

\bibitem{Diaz97}
J.I.~D\'iaz.
\newblock On the mathematical treatment of energy balance climate models.
\newblock In {\em The mathematics of models for climatology and environment
  ({P}uerto de la {C}ruz, 1995)}, volume~48 of {\em NATO ASI Ser. I Glob.
  Environ. Change}, pages 217--251. Springer, Berlin, 1997.

\bibitem{Diaz07}
J.I.~D\'iaz and L.~Tello.
\newblock A 2{D} climate energy balance model coupled with a 3{D} deep ocean
  model.
\newblock In {\em Proceedings of the 2006 {I}nternational {C}onference in honor
  of {J}acqueline {F}leckinger}, volume~16 of {\em Electron. J. Differ. Equ.
  Conf.}, pages 129--135. Texas State Univ.--San Marcos, Dept. Math., San
  Marcos, TX, 2007.

\bibitem{Temam2001}
B.~Ewald and R.~Temam.
\newblock Maximum principles for the primitive equations of the atmosphere.
\newblock {\em Discrete Contin. Dynam. Systems}, 7(2):343--362, 2001.

\bibitem{Fraedrich2001}
K.~Fraedrich.
\newblock Simple climate models.
\newblock In {\em Stochastic climate models ({C}horin, 1999)}, volume~49 of
  {\em Progr. Probab.}, pages 65--100. Birkh\"auser, Basel, 2001.

\bibitem{Ghil1976}
M.~Ghil.
\newblock Climate stability for a {S}ellers-type model.
\newblock {\em J. Atmospheric Sci.}, 33(1):3--20, 1976.

\bibitem{GGHHK:20b}
Y.~Giga, M.~Gries, M.~Hieber, A.~Hussein, and T.~Kashiwabara.
\newblock Analyticity of solutions to the primitive equations.
\newblock {\em Math. Nachr.}, 293:284--304, 2020.

\bibitem{Vicol2014}
N.~Glatt-Holtz, I.~Kukavica, V.~Vicol, and M.~Ziane.
\newblock Existence and regularity of invariant measures for the three
  dimensional stochastic primitive equations.
\newblock {\em J. Math. Phys.}, 55(5):051504, 34, 2014.

\bibitem{Guo2009}
B.~Guo and D.~Huang.
\newblock 3{D} stochastic primitive equations of the large-scale ocean: global
  well-posedness and attractors.
\newblock {\em Comm. Math. Phys.}, 286(2):697--723, 2009.

\bibitem{Hasselmann1976}
K.~Hasselmann.
\newblock Stochastic climate models part {I}. {T}heory.
\newblock {\em Tellus}, 28(6):473--485, 1976.

\bibitem{Hieber2016}
M.~Hieber, A.~Hussein, and T.~Kashiwabara.
\newblock Global strong {$\rL^p$} well-posedness of the 3{D} primitive
  equations with heat and salinity diffusion.
\newblock {\em J. Differential Equations}, 261(12):6950--6981, 2016.

\bibitem{HK:16}
M.~Hieber and T.~Kashiwabara.
\newblock Global strong well-posedness of the three dimensional primitive
  equations in {$\rL^p$}-spaces.
\newblock {\em Arch. Ration. Mech. Anal.}, 221(3):1077--1115, 2016.

\bibitem{Imkeller2001}
P.~Imkeller and J.~von Storch, editors.
\newblock {\em Stochastic climate models}, volume~49 of {\em Progress in
  Probability}. Birkh\"auser Verlag, Basel, 2001.

\bibitem{KZ:07a}
I.~Kukavica and M.~Ziane.
\newblock On the regularity of the primitive equations of the ocean.
\newblock {\em Nonlinearity}, 20(12):2739--2753, 2007.

\bibitem{LTW:92a}
J.~Lions, R.~Temam, and S.~Wang.
\newblock New formulations of the primitive equations of atmosphere and
  applications.
\newblock {\em Nonlinearity}, 5(2):237--288, 1992.

\bibitem{LTW:92b}
J.~Lions, R.~Temam, and S.~Wang.
\newblock On the equations of the large-scale ocean.
\newblock {\em Nonlinearity}, 5(5):1007, 1992.

\bibitem{LTW:95}
J.~Lions, R.~Temam, and S.~Wang.
\newblock Mathematical theory for the coupled atmosphere-ocean models (cao
  iii).
\newblock {\em J. Math. Pures Appl.}, 74:105--163, 1995.

\bibitem{Majda2003}
A.~Majda.
\newblock {\em Introduction to {PDE}s and Waves for the Atmosphere and Ocean},
  volume~9 of {\em Courant Lecture Notes in Mathematics}.
\newblock American Mathematical Society, 2003.

\bibitem{Watts1990}
M.~Morantine and R.~Watts.
\newblock Upwelling diffusion climate models: Analytical solutions for
  radiative and upwelling forcing.
\newblock {\em Journal of Geophysical Research: Atmospheres},
  95(D6):7563--7571, 1990.

\bibitem{MR3019280}
T.~Nau.
\newblock The {L}aplacian on cylindrical domains.
\newblock {\em Integral Equations Operator Theory}, 75(3):409--431, 2013.

\bibitem{Veraar2012b}
J.~van Neerven, M.~Veraar, and L.~Weis.
\newblock Maximal {$L^p$}-regularity for stochastic evolution equations.
\newblock {\em SIAM J. Math. Anal.}, 44(3):1372--1414, 2012.

\bibitem{Veraar2012a}
J.~van Neerven, M.~Veraar, and L.~Weis.
\newblock Stochastic maximal {$L^p$}-regularity.
\newblock {\em Ann. Probab.}, 40(2):788--812, 2012.

\bibitem{North17}
G.~North and K.~Kim.
\newblock {\em Energy balance climate models}.
\newblock Wiley Series in Atmospheric Physics and Remote Sensing". John Wiley
  \& Sons, 2017.

\bibitem{Pruss2006}
J.~Pr\"uss, R.~Racke, and S.~Zheng.
\newblock Maximal regularity and asymptotic behavior of solutions for the
  {C}ahn-{H}illiard equation with dynamic boundary conditions.
\newblock {\em Ann. Mat. Pura Appl. (4)}, 185(4):627--648, 2006.

\bibitem{Sellers1969}
W.~Sellers.
\newblock A global climatic model based on the energy balance of the
  earth-atmosphere system.
\newblock {\em Journal of Applied Meteorology and Climatology}, 8(3):392--400,
  1969.

\bibitem{Vallis2017}
G.~Vallis.
\newblock {\em Atmospheric and Oceanic Fluid Dynamics: Fundamentals and
  Large-Scale Circulation}.
\newblock Cambridge University Press, second edition, 2017.

\bibitem{Wang2009}
W.~Wang and J.~Duan.
\newblock Reductions and deviations for stochastic partial differential
  equations under fast dynamical boundary conditions.
\newblock {\em Stoch. Anal. Appl.}, 27(3):431--459, 2009.

\end{thebibliography}

\end{document}